\documentclass[11pt,a4paper,reqno]{amsart}

\usepackage{amsmath,amssymb,graphics,epsfig,color,enumerate,psfrag,mdframed}
\usepackage{dsfont}  
\usepackage{verbatim}
\newtheorem{Theorem}{Theorem}[section]

\newtheorem{TheoremA}{Theorem}
\newtheorem{Lemma}[Theorem]{Lemma}
\newtheorem{Proposition}[Theorem]{Proposition}
\newtheorem{Corollary}[Theorem]{Corollary}
\newtheorem{Remark}[Theorem]{Remark}

\newtheorem{Conclusion}[Theorem]{Conclusion}

 \definecolor{darkgreen}{rgb}{0,0.4,0}

\definecolor{light}{gray}{.9}

\newcommand{\cC}{\ensuremath{\mathcal C}}

\newcommand{\cE}{\ensuremath{\mathcal E}}

\newcommand{\cG}{\ensuremath{\mathcal G}}

\newcommand{\cL}{\ensuremath{\mathcal L}}

\newcommand{\cT}{\ensuremath{\mathcal T}}

\newcommand{\cV}{\ensuremath{\mathcal V}}

\newcommand{\bbE}{{\ensuremath{\mathbb E}} }

\newcommand{\bbN}{{\ensuremath{\mathbb N}} }

\newcommand{\bbP}{{\ensuremath{\mathbb P}} }

\newcommand{\bbR}{{\ensuremath{\mathbb R}} }

\newcommand{\bbZ}{{\ensuremath{\mathbb Z}} }

\let\a=\alpha \let\b=\beta   \let\d=\delta  \let\e=\varepsilon
 \let\g=\gamma       \let\l=\lambda
   \let\n=\nu         
\let\r=\rho  \let\s=\sigma \let\t=\tau   
  
\let\D=\Delta   \let\G=\Gamma  \let\L=\Lambda 
 \let\P=\Pi     

\newcommand{\ds}{\displaystyle}

\author[A.\ Faggionato]{Alessandra Faggionato}
\address{Alessandra Faggionato.
  Dipartimento di Matematica, Universit\`a di Roma `La Sapienza'
  P.le Aldo Moro 2, 00185 Roma, Italy}
\email{faggiona@mat.uniroma1.it}

\author[V. Silvestri]{Vittoria Silvestri}
\address{Vittoria Silvestri. 
DAMTP, Centre for Mathematical Sciences,
Wilberforce Road,
Cambridge,
CB3 0WA,
United Kingdom.}
\email{V.Silvestri@maths.cam.ac.uk}

\begin{document}

\begin{abstract}    We consider random walks on quasi one dimensional  lattices, as introduced in \cite{FS}. This mathematical setting covers a large class of discrete kinetic models for non--cooperative molecular motors on periodic tracks.  We derive general formulas for the asymptotic velocity and diffusion coefficient, and we  show how to reduce their computation to  suitable linear systems of the same degree of a single  fundamental cell, with possible linear chain removals. We apply the above results to special families of kinetic models, also 
catching some errors in the biophysics literature.


\end{abstract}

\title[Discrete kinetic models for molecular motors]{Discrete kinetic models for molecular motors: asymptotic velocity and gaussian fluctuations}

\maketitle

\section{Introduction}
Molecular motors are proteins working   inside the cell as nanomachines \cite{H}. They usually convert chemical energy coming from ATP  to produce mechanical work fundamental for cargo transport, cell division, genetic transcription, muscle contraction,... We concentrate here on the large family of  molecular motors working in a non--cooperative way and  moving along cytoskeletal filaments, which are given by polarized homogeneous polymers.

Molecular motors have been and still are object of intensive study in biology and biophysics. Two fundamental paradigms have been proposed for their modelization. In the so called \emph{Brownian ratchet} model \cite{JAP,PJA,Re} the dynamics of the molecular motor is given by a one--dimensional diffusion in a periodic but typically asymmetric potential, which can switch to a different potential at random times (switching diffusion). 
 The other paradigm, on which we concentrate here,   is given by continuous time random walks (also with non exponential waiting times)  on quasi linear graphs having a periodic structure \cite{FK1,FK2,K2,KF1,KF2,KF3}.  As in \cite{FS} we call these graphs \emph{quasi 1d lattices}, they are obtained by gluing together several copies of a fundamental cell
in a linear fashion. The geometric  complexity of the fundamental cell reflects the possible conformational transformations of the molecular motor in its mechanochemical cycle. The simplest example is given by a   random walk\footnote{If not stated  otherwise, by ``random walk'' we mean a Markovian random walk, hence  with exponential waiting times.} on $\bbZ$ with periodic jump rates (in this case the fundamental cell is given by an interval with $N$ sites, $N$ being the periodicity), while random walks on  other classes of   quasi 1d lattices (parallel--chain models and divided--pathway models)  and also with non--exponential waiting times have been studied motivated by experimental evidence of a richer structure (cf. e.g. \cite{DK1,DK2,K1,KF1}).

 Still before the study of molecular motors,  both the asymptotic velocity and  gaussian fluctuations for the  random walk on $\bbZ$ with periodic jump rates have been obtained from \cite{D} under a  suitable \emph{Ansatz}.  Generalizing the same Ansatz, formulas have been given in \cite{DK1,DK2,K1} for the asymptotic velocity and  gaussian fluctuations of  parallel--chain models and divided--pathway models.
In \cite{TF} the authors have  considered  random walks on a quasi 1d lattice and, by first--passage time arguments,   have investigated the asymptotic velocity.

 We treat random walks (also with non--exponential waiting times) on quasi 1d lattices in full generality and show that  their analysis reduces to the study of random time changes of sums of i.i.d. random variables.  Indeed,  all relevant information concerning the position of the random walk are encoded in a suitable random walk on $\bbZ$ (that  we call \emph{skeleton process}) with nearest neighbor jumps, spatially homogenous rates and  typically non--exponential holding times.   The skeleton process belongs to a larger class of stochastic processes  $(Z_t)_{t \in \bbR_+}$ obtained by 
considering a sequence $( w_i, \t_i )_{i \geq 1}$ of i.i.d.   2d vectors with values in
$\bbR \times (0,+\infty)$, defining $W_m:= \sum _{i=1}^m w_i$, $
\cT_m:=\sum_{i=1}^m \t_i$ for $m \geq 0$ integer, and setting  $Z_t:= W _{ \max \{ m\geq 0 :\, \cT_m \leq t \}}$. 
 Sums of i.i.d. random variables have many nice properties  and, assuming $\bbE(\t_i ^2)<\infty$,  random time changes behave well for  what concerns  law of large numbers  and invariance principle (while they are troublesome for what concerns    large deviations  \cite{FS}).  
Indeed, we derive the law of large numbers  and  the invariance principle
    for the above processes  $(Z_t)_{t \in \bbR_+}$ (see Appendix \ref{codina}), and   therefore Theorem \ref{zecchino} (law of large numbers  and invariance principle
 for the skeleton process) is a corollary of this general result.

We also treat computability issues. In Section \ref{marina} we   show that in the computation of the velocity one can simplify the  quasi 1d lattice by removing linear chains and substituting them with a single edge. 
A similar  reduction has been derived in \cite{TF}  by   different arguments based on 
expected occupation times.  Moreover, in Section \ref{1prigione}  we show that both the velocity $v$ and  the diffusion coefficient $\s^2$   
 can be computed by means of linear systems of  the same degree of the fundamental cell 
 of the quasi 1d lattice.
  
  Finally, we apply the above results to derive the asymptotic velocity and diffusion coefficient in specific classes of 1d quasi lattices: the $N$--periodic random walk  on $\bbZ$ \cite{D} (cf. Section \ref{esempi1}) and the parallel--chain model \cite{K1} (cf. Section \ref{esempi3}). 
Our formulas for the asymptotic velocity confirm the expressions obtained in \cite{D,K1}. Formulas for the asymptotic diffusion coefficient are rather complex.  For the $2$--periodic random walk on $\bbZ$  we show that our formula is in agreement with the expression derived in \cite{D}.  Our formula of the asymptotic diffusion coefficient   contradicts the expression derived in \cite{K1}  for the parallel--chain model by generalizing Derrida's Ansatz. In Section \ref{confronto} we consider  special parallel--chain models and explain why the expression obtained in \cite{K1} contradicts some general principle.  Our discussion has led to a revision of \cite{K1} from the author who has indicated \cite{K} that formulas (26) and (28) should be corrected by taking   the second sum in  the third  addendum of (26) and (27) starting from  $i=1$ and not from $i=0$.
We also point out that in the last sum of  (28) in \cite{K1}  the term $b_j^{(1)}$ should be replaced by $ b_i^{(1)}$. For small chains, making these  corrections, our general formulas confirm the corrected version of 
\cite{K1}.

 The random walk on $\bbZ$ with $N$--periodic rates treated in \cite{D} has been considered as a reference model  in the study of several discrete models for molecular motors. We point out   that one cannot expect that the mathematical features of this random walk are  universal. For example, as discussed in \cite{FS,FS1}, this random walk has the property that, starting from $0$, the time needed to hit the set $\{-N,N\}$ is independent from the precise site, $-N$  or $N$, visited when arriving at $\{-N,N\}$. This mathematical property, which corresponds to  a  Gallavotti--Cohen type symmetry for the  large deviations of the random walk,  is shared only inside a special class of discrete kinetic models (not  including for example the parallel-chain model or the divided-pathway  model).
 For more results  on fluctuation theorems in discrete kinetic model of molecular motors we refer the interested reader to  \cite{FS,FS1} and references therein.

\section{Random walks on quasi one dimensional lattices} \label{quasi}

In this section we define the quasi one dimensional lattices and the stochastic processes we will focus on. 

\subsection{Quasi one dimensional lattices} We start by defining  quasi 1d lattices.
Consider a finite oriented graph $G= (V,E)$, $V$ being the set of vertices and $E$ being the set of oriented edges,  $E\subset \{ (v,w)\,:\, v \not = w \text{ in } V \}$.  We fix  in $V$ two vertices $\underline{v}, \overline{v}$. 
 We assume that  the oriented graph $G$  is connected, i.e. for any $v,w \in V$ there is an oriented path in $G$ from $v $ to $w$. 
 The quasi 1d lattice $\cG$ associated to the triple $\bigl(G,\underline{v},\overline{v}\bigr)$ is the oriented graph obtained by gluing together countable copies of $G$ such that   the point $\overline{v}$ of one copy is identified with the point $\underline{v}$ of the next copy. To formalize the definition,  we define $\cG$ as $\cG=(\cV, \cE)$ with vertex set $\cV$ and edge set $\cE$ as follows
 (see Fig. \ref{pocoyo100}):
\begin{align*}
& \cV:=\left\{ v_n :=(v,n) \in (V \setminus \{\overline{v}\})\times \bbZ\right\} \,\\
& \cE :=\cE_1 \cup \cE_2 \cup \cE_3\,,
\end{align*}
where 
\begin{align*}
& \cE_1:= \left\{ (v_n, w_n)\,:\, (v,w) \in E \,,\; n \in \bbZ \right \} \,,\\
& \cE_2:= \cup _{n \in \bbZ} \left\{ (v_n, \underline{v}_{n+1} )\,:\, (v,\overline{v}) \in E \right\} \,,\\
&\cE_3:= \cup_{n \in \bbZ}  \left\{ (\underline{v}_{n+1},v_n )\,:\, (\overline{v},v) \in E \right\} \,.
\end{align*}

\begin{figure}[!ht]
    \begin{center}
     \centering
  \mbox{\hbox{
  \includegraphics[width=0.9\textwidth]{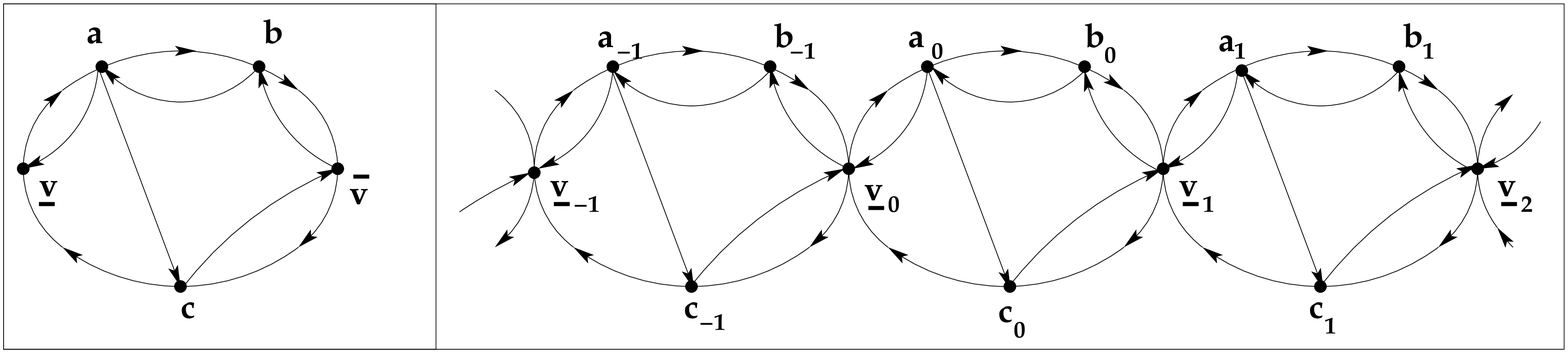}}}
            \end{center}
            \caption{The  graph $G=(V,E)$ with vertices $\underline{v}, \overline{v}$ (left) and the associated quasi 1d lattice $\cG= (\cV, \cE)$ (right) }\label{pocoyo100}
  \end{figure}
In what follows, it will be convenient to  simplify the notation setting 
\begin{equation*}
n_*:= \underline{v}_n\,, \qquad n \in \bbZ\,.
\end{equation*}

\subsection{The process  $X$ and the skeleton process $X^*$}

 Given a  quasi 1d lattice $\cG$ associated to the triple $\bigl(G,\underline{v},\overline{v}\bigr)$,  we consider  a generic stochastic process 
  $X=( X_t)_{t \in \bbR_+}$  with state space $\cV$, jumping only along the oriented edges in $\cE$, 
whose dynamical rules are $\cT$--invariant and such that, when arrived at a site $n_*$, it  looses memory of its past. We refer to   \cite{FS}[Def. 2.1] for a precise mathematical definition. The  key example  is given by a  continuous time random walk    $X=( X_t)_{t \in \bbR_+}$
    with state space $\cV$, starting at $0_*= \underline{v}_0$ and with positive  jump rates $r(x,y)$,  $(x,y) \in \cE$, invariant by cell shift:
 \begin{equation}\label{simm1}
r(x,y)= r( \cT x, \cT y)\,, \qquad \qquad \cT(v_n):= v_{n+1}\,.
 \end{equation}
In particular, the random walk $X$ waits at $x$ an exponential time with inverse mean $r(x):= \sum _{y:(x,y) \in \cE} r(x,y)$ and then   jumps to a neighboring vertex $y$ with probability $r(x,y)/r(x)$. One can consider also non exponential holding times,  as in \cite{KF1}.

  In the applications to molecular motors, 
each site  $n_*$ corresponds to a spot in the $n$--monomer of the polymeric filament where the molecular motor can bind. The other states $v_n$ correspond to intermediate conformational states that  the molecular motor
achieves in its mechanochemical transformations, which are described by jumps along edges in $\cE$.  In particular,  states  $v_n$  do not  encode only a spatial position and jumps do not necessarily correspond to spatial jumps.
\smallskip

We   now introduce  a coarse--grained version of $X$, given by the so--called \emph{skeleton process} $X^*= (X_t^*)_{t \in \bbR_+}$, which 
 has values in   $\bbZ$ and  records the last visited 
state of the form  $n_*$ up to time $t$.
More precisely, 
it is  defined as  $ X_t^*:= \Phi(X_\tau)$ where $\Phi(n_*)=n$ and 
 $$\tau := \sup \left\{ s \in [0,t]: X_t=n_* \text{ for some } n \in \bbZ\right\}$$
 (see Fig. \ref{esempietto_fig}). In the applications to molecular motors,  the skeleton process contains  all the relevant information, indeed it allows to determine the position of the molecular motor up to an error of the same order of the monomer size. 
 
 \begin{figure}[!ht]
     \begin{center}
     \centering
  \mbox{\hbox{
  \includegraphics[width=0.7\textwidth]{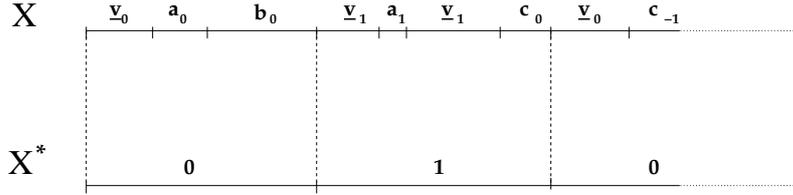}}}
         \end{center}
         \caption{Example of a trajectory  $X= (X_t)_{t \in \bbR_+}$ and the associated trajectory  $X^*= (X_t^*)_{t \in \bbR_+}$ referred to the quasi 1d lattice $\cG$ of Fig. \ref{pocoyo100}}\label{esempietto_fig}
  \end{figure}

\section{Asymptotic velocity and gaussian fluctuations}\label{panettone}
Considering the process $X$ starting at $0_*$
 let  $S$ be the random time defined as \begin{equation}\label{pietro}S
  := \inf \left\{ t \geq 0 \,:\, X_t  \in \{-1_*,1_*\}    \right \}\,.
  \end{equation}
  In what follows we assume that   $\bbE(S^2)< \infty$ (this assumption is always verified for continuous time random walks with exponential waiting times \cite{FS}[Lemma 2.2]).
 
\begin{TheoremA} \label{zecchino}
Consider the process  $( X_t )_{t \in \bbR_+}$ starting at $0_*$. Then   the skeleton process $(X^*_t )_{t \in \bbR_+}$ satisfies the following properties:
\begin{itemize}

\item[(i)] ({\sl Asymptotic velocity})  Almost surely 
\begin{equation}
\lim _{t \to \infty} \frac{X^*_t}{t}=
 \frac{  \bbP (X_S=1_*)-\bbP (X_S=-1_*)}{\bbE(S)}=:v\,.\label{trenino}
\end{equation}

\item[(ii)] ({\sl  Gaussian fluctuations})
Given $n\in \bbN $ define the rescaled process 
$$ B^{(n)} _t := \frac{ X_{ nt }^* - vnt}{\sqrt{n}}  \,, \qquad t \in \bbR_+\,,
$$ with paths in the Skohorod space\footnote{The Skohorod space   $ D( \bbR_+; \bbR)$ is given by paths from $\bbR_+$ to $\bbR$, which are  right continuous and have left limits. It is endowed with the so called Skohorod topology \cite{B}. 
Alternatively,   one could define $B_t^{(n)}$ by rescaling and linear interpolation, and replace $D( \bbR_+;\bbR) $ by  the space of continuous paths $C (\bbR_+; \bbR)$ where convergence corresponds to uniform convergence on compact intervals. The invariance principle remains true in this alternative setting.}
  $ D( \bbR_+; \bbR)$.  Then as $n\to \infty$ the rescaled process $\bigl(B^{(n)}_t\bigr)_{t \in \bbR_+}$ weakly converges to a Brownian motion on $\bbR$ with diffusion constant
\begin{equation}\label{diffondo1}
\s^2:= \frac{ {\rm Var} (X^*_S-v S)}{ \bbE (S) }\,.
\end{equation}
\end{itemize}
\end{TheoremA}

The above result is a corollary of Theorem \ref{teo1} (presented in Appendix \ref{codina} together with its proof) concerning the law of large numbers and the invariance principle for random time changes of  cumulated processes.

\begin{Remark}
Considering random walks with exponential holding times, the random variables $X_S^*$ and $S$ are independent if the fundamental graph $G$ is $(\underline{v}, \overline{v})$--minimal as defined in \cite{FS}[Section 3.2]. For such random walks (including nearest--neighbor random walks on $\bbZ$ with periodic rates) the variance in \eqref{diffondo1} becomes  ${\rm Var} (X^*_S)+v^2  {\rm Var} ( S)$.
\end{Remark}
\section{Computation of the asymptotic velocity $v$ }\label{marina}

In this section we explain how the asymptotic velocity $v$ can be obtained by solving suitable 
linear systems, whose degree order can be reduced by removal of linear chains, when restricting to  Markovian random walks (hence in this section $X$ is a random walk with exponential waiting times satisfying \ref{simm1}).
\subsection{Reduction to a 2-cells random walk}\label{sanremo} We
define $\tilde{X}:=(\tilde X_t)_{t \in \bbR_+}$ by
$ \tilde{X}_t= X_{S \wedge t}$. Simply,  $\tilde X$ differs from $X$ by the fact that the states $-1_*,1_*$ are  absorbing (recall that $\tilde{X}_0= X_0=0_*$). 
Trivially $\tilde{X}$ is again a random walk with state space given by the finite set $\tilde V:=\{ v_n \in \cV \,:\, v \in V\,, n=-1,0 \} \cup \{1_*\}$,  obtained by gluing two fundamental cells. The associated  graph $\tilde G =( \tilde V, \tilde E)$ is such that no arrows exits from $-1_*,1_*$, while for all other edges  the  jump rates are the same as for $X$ (see Fig. \ref{pocoyo2}).
\begin{figure}[!ht]    \begin{center}
     \centering
  \mbox{\hbox{
  \includegraphics[width=0.4\textwidth]{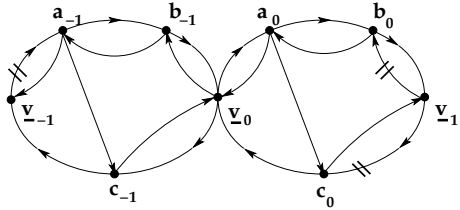}}}
            \end{center}
            \caption{Graph $\tilde G= (\tilde V, \tilde E)$ associated to the random walk $\tilde X$,  edges 
           marked with double lines have to be suppressed. For all other edges the jump rates are given by $r(\cdot, \cdot)$}\label{pocoyo2}
  \end{figure}

  We note that 
$ \bbP  (X_S=\pm 1_*) = \bbP \bigl( \tilde{X}_t = \pm 1_* \text{ eventually in } t \bigr)$.
Being a hitting probability, $ \bbP  (X_S=\pm 1_*)$ can be computed by solving the finite linear system in Theorem 3.3.1 of \cite{N}. In particular, setting 
\[ \phi(x):=  \bbP \bigl( \tilde{X}_t =   1_* \text{ eventually in } t \, | \, \tilde X_0= x\bigr)\,, \qquad x \in \tilde V\,,
\]
the family of  hitting probabilities $ \{ \phi(x) \}_{x \in \tilde V}$  is the minimal nonnegative solution of the linear system
\begin{equation}\label{febbre1}
\begin{cases}
\sum _{y \in \tilde V: (x,y) \in \tilde E}\tilde  r(x,y)\bigl[ \phi(x)- \phi(y)\bigr]=0 \,, \qquad x \in \tilde V\setminus\{-1_*,1_*\}\,,\\
\phi(-1_*)=0\,,\\
\phi(1_*)=1\,,
\end{cases}
\end{equation}
where $\tilde r(x,y):= r(x,y)$ for all $(x,y) \in \tilde E $.  Then the term $\bbP(X_S=1_*)- \bbP(X_S=-1_*)$ appearing in formula \eqref{trenino} for the asymptotic velocity equals  $2\phi(0_*)-1$.

  Moreover we note that $S$ has the same law of 
$\tilde S:= \inf \bigl\{ t \geq 0 \,:\, \tilde{X}_t  \in \{-1_*,1_*\}    \bigr \}$.
Hence, the expected value $\bbE(S)$ (appearing in formula \eqref{trenino} for the asymptotic velocity) can be computed by   solving the finite linear system in Theorem 3.3.3 of \cite{N}.  
Then, setting 
\[ \psi (x):= \bbE \bigl( \tilde S\,\big |\,  \tilde X_0=x\bigr) \,, \qquad x \in \tilde V\,,
\] we  have $\bbE(S) = \psi(0_*)$ and
the family of  expected hitting times $ \{ \psi(x) \}_{x \in \tilde V}$  is the minimal nonnegative solution of the linear system
\begin{equation}\label{febbre2}
\begin{cases}
 \sum _{ y \in \tilde V\,:\, (x,y) \in \tilde E} \tilde r(x,y)\bigl[ \psi(x)-\psi(y) \bigr]  =1  \,, \qquad x \in \tilde V\setminus\{-1_*,1_*\}\,,\\
\psi(-1_*)=\psi(1_*)=0\,.
\end{cases}
\end{equation}

\subsection{Removal of linear chains}\label{rimozione}
We present a method to reduce the dimension of the above  linear systems \eqref{febbre1} and \eqref{febbre2} by removing linear chains from the graph $\tilde G$. This removal is particularly useful for studying the asymptotic velocity of several kinetic models for molecular motors  presenting linear chains inside (cf. the periodic linear model, parallel--chain model, divided--path model discussed below).  A similar  reduction has been derived in \cite{TF}  by   different arguments based on 
expected occupation times. Propositions \ref{virulino1} and  \ref{virulino2} below  give a direct method to compute quantities that in \cite{TF} are derived by solving more linear systems.
 The method is completely general, and works for any Markov chain with absorbing states (as the one presented in Section \ref{1prigione}, anyway for simplicity we refer only to the random walk $\tilde X$ on $\tilde G$).

 We say that the sequence of  vertices  $\g=(x_0,x_1,\dots, x_n)$ in $\tilde V$  forms a linear chain in $\tilde G$  if $n\geq 2$ and,  for any $i=0,1, \dots, n-1$, 
 the only edges in $\tilde E$ exiting from $x_i$ are $(x_i, x_{i\pm 1})$ and  the only edges in $\tilde E$  entering into $x_i$ are $(x_{i\pm 1}, x_i)$ (see Fig. \ref{uovo}).  Note that we have excluded the trivial case $n=1$. The vertices $x_1, x_2, \dots, x_{n-1}$ are called \emph{intermediate vertices}, while 
  the  vertices $x_0,x_n$ are called \emph{extremal vertices} ($x_0$ initial, $x_n$ final).  Moreover, the edges $(x_i, x_{i+1
})$ and  $(x_{i+1}, x_i)$, as $i $ varies among $0,1, \dots, n-1$, are called \emph{edges of the chain}.  Note that if  the sequence $\g=(x_0,x_1,\dots, x_n)$ forms a  linear chain in $\tilde G$, then also  the inverted one, i.e. 
 $\bar \g:=(x_n,x_{n-1},\dots, x_0)$, forms a  linear chain in $\tilde G$.  
 Introducing the short notation  \begin{equation}\label{aereo}
   r_i^- := r(x_i, x_{i-1}) \,, \qquad r_i^+ := r(x_i,x_{i+1}) \,, \qquad 1\leq i \leq n-1\,,
 \end{equation}
 given a linear chain $\g$ as above we define
 \begin{align}
 \G (\g)& := 1+\sum_{i=1}^{n-1} \frac{ r_1^-}{r_1^+} \frac{ r_2^-}{r_2^+}  \cdots\frac{ r_i^-}{r_i^+}\,, \label{alexian}\\
 c(\g)& :=
  \sum_{1\leq k \leq j \leq n-1} \frac{1}{r_k^+} \frac{r_{k+1}^-}{r_{k+1}^+} \cdots \frac{r_j-}{r_j^+}  
 \,,\label{rachele}
 \end{align} 
 where we set $\frac{1}{r_k^+} \frac{r_{k+1}^-}{r_{k+1}^+} \cdots \frac{r_j-}{r_j^+}  = \frac{1}{r_k^+}$ when $k=j$.
 Note that the definition of $\G(\g)$ depends on the orientation of $\g$. Moreover, note that
  \begin{equation}\label{supersonico}\G (\bar \g) := 1+\sum_{i=1}^{n-1} 
  \frac{ r_{n-1}^+}{r_{n-1}^-} \frac{ r_{n-2}^+}{ r_{n-2}^-}  \cdots 
 \frac{ r _{n-i}^+}{r_{n-i}^-} = \G(\g) \Big( \frac{ r_1^-}{r_1^+} \frac{ r_2^-}{r_2^+}  \cdots\frac{ r_{n-1}^-}{r_{n-1}^+}\Big)^{-1}\,.
 \end{equation}
 
 \begin{figure}[!ht]
    \begin{center}
     \centering
  \mbox{\hbox{
  \includegraphics[width=0.5\textwidth]{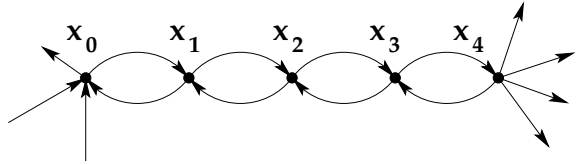}}}
            \end{center}
            \caption{A linear chain in $\tilde G$ with $n=4$.}\label{uovo}
  \end{figure}

  Fix a family $\cC$ of linear chains  and set $\bar \cC:=\{ \bar \g: \g \in \cC\}$. We assume that  (i) if $\g \in \cC$ then $\bar \g \not \in \cC$, (ii) given $\g \not= \g'$ in $\cC\cup \bar \cC$, the chain $\g$ is not a subchain of $\g'$ and vice versa (this second condition is to avoid 
 trivial steps in the algorithmic definition of $\bar G$ given below). 
  Consider the  new graph $\bar G:= (\bar V, \bar E)$ obtained from $\tilde G= ( \tilde V, \tilde E)$ 
  as follows. 
 Take any  pair  of  vertices  being the extremal vertices of some linear chain in $\cC\cup \bar \cC$. Order arbitrarily the two vertices: call $x$ the first one and $y$ the second one.
      Let $\g^{(1)}, \g^{(2)}, \dots , \g^{(k)}$ be the linear chains in $\cC \cup\bar \cC$ from $x$ to $y$.Write $\g^{(1)}=(x^{(1) }_0=x, x^{(1)}_1, \dots, x^{(1)} _{n_1}=y ) 
  $, $\g^{(2)}=(x^{(2) }_0=x, x^{(2)}_1, \dots, x^{(2)} _{n_2}=y ) 
  $, and so on. 
Then for each linear chain    $\g^{(i)}$ remove from $\tilde G$ all its intermediate vertices and its edges. Then, add the edges $(x,y)$, $(y,x)$ (if they are not present yet).  At this point define 
\begin{align}
&  \bar r(x,y)= r(x,y)\mathds{1}( (x,y) \in \tilde E)  + \sum _{s=1}^k \frac{ r(x, x^{(s)}_1) }{ \G\bigl( \g^{(s) }\bigr) }\,, \label{rocco_hunt_1}\\
&  \bar r(y,x)= r(y,x)\mathds{1}(( y,x) \in \tilde E)  + \sum _{s=1}^k \frac{ r(y, x^{(s)}_{n_s-1} ) }{ \G\bigl(\bar  \g^{(s) }\bigr) }\label{rocco_hunt_2}\,.
\end{align}  Above, the characteristic function $\mathds{1}( (x,y) \in \tilde E) $  equals $1$  if $(x,y) \in \tilde E$, otherwise it equals $0$. Similarly for $\mathds{1}(( y,x) \in \tilde E) $. It is simple to verify that the above definitions  \eqref{rocco_hunt_1} and \eqref{rocco_hunt_2} are well posed, since they do not depend on the chosen order between $x$ and $y$.
  After applying the above procedure for any   pair  of  vertices    being the extremal vertices of some linear chain in $\cC\cup \bar \cC$, call $\bar G= (\bar V, \bar E)$ the resulting graph and 
  set 
 $\bar r(x,y)= \tilde r (x,y)= r(x,y)$ for all edges  $(x,y)\in \bar E\cap \tilde E$  such that $x,y$ are not the extremal points of some linear chain in $\cC\cup \bar \cC$.

 \begin{Proposition}\label{virulino1}
 The family   of hitting probabilities $ \{ \phi(x) \}_{x \in \bar V}$  (cf. Section \ref{sanremo}) is the minimal nonnegative solution of the linear system
\begin{equation}\label{febbre1bis}
\begin{cases}
 \sum _{y \in  \bar V: (x,y) \in \bar E}\bar   r(x,y)\bigl[ \phi(x)-\phi(y)\bigr]=0\,, \qquad x \in \bar V\,,\\
\phi(-1_*)=0\,,\\
\phi(1_*)=1\,,
\end{cases}
\end{equation}
 \end{Proposition}
 Note that, by Theorem 3.3.1 of \cite{N}, given $x \in \bar V$ the  quantity $\phi(x)$ equals the hitting probability to $1_*$ for the random walk on $\bar G$ with jump rates $\bar r(\cdot, \cdot)$. 
  
  \begin{Proposition}\label{virulino2}
 The family   of expected hitting times $ \{ \psi(x) \}_{x \in \bar V}$  (cf.  Section \ref{sanremo})  is the minimal nonnegative solution of the linear system 
 \begin{equation}\label{febbre2bis}
\begin{cases}
\sum _{y \in  \bar V: (x,y) \in \bar E}\bar   r(x,y)\bigl[ \psi(x)- \psi(y)\bigr] = 1+ c(x) \,, \qquad x \in \bar V\,,\\
\psi(-1_*)=0\,,\\
\psi(1_*)=0\,,
\end{cases}
\end{equation}
 where  the constant $c(x) $ is computed as follows:  let $\g^{(1)}, \dots, \g^{(m)}$ be the linear chains in $\cC$ exiting from $x$ and let $x_1^{(s)}$ be the first intermediate point after $x$ in $\g^{(s)}$. Then 
\begin{equation}\label{torlototto}
  c(x)=  \sum_{s=1}^m \frac{\tilde r (x, x_1^{ (s)}) c\bigl(\g^{(s)} \bigr)}{\G(\g^{(s)})}\,,
  \end{equation}
 where $c\bigl(\g^{(s)} \bigr)$ is defined as in \eqref{rachele}.
 \end{Proposition} 
 
\begin{Conclusion} Taking $\cC$ such that $0_* \in \bar V$, from Theorem \ref{zecchino} we get that the asymptotic velocity $v= \left(2 \phi(0_*) -1\right) / \psi(0_*)$ can be computed solving  the above linear systems \eqref{febbre1bis} and \eqref{febbre2bis} parametrized by $\bar V$.
\end{Conclusion}

 The proof of the above Propositions \ref{virulino1} and \ref{virulino2} is postponed to  Section \ref{EA}.


\section{Computability issues: cell reduction}\label{1prigione}
In the previous section 
 we have dealt with  the  random walk $\tilde{X}$ on   two fundamental cells suitably glued together and we have concentrated on  the velocity.  
 In general, both the velocity $v$ and  the diffusion coefficient $\s^2$   
 can be computed by means of linear systems. To this aim we recall some general fact in probability theory that extends Theorems 3.3.1 and 3.3.3 of \cite{N}.
 
Let $(W_t)_{t \in \bbR_+}$ be a Markov chain (with exponential holding times) on a finite state space $\aleph$  and let $c(x,y)$ denote the probability rate for a jump from $x$ to $y \not =x$, i.e. $\bbP( W_{t+dt}=y|X_t=x)= c(x,y) dt + o(dt)$. We denote by $\cL$ the Markov generator of the chain:
\[ \cL f(x) = \sum _{y \not =x} c(x,y) \left( f(y)-f(x) \right) \,, \qquad f : \aleph \mapsto \bbR\,.
\]
We write $c(x)$ for the holding parameter at $x$: $c(x) = \sum _{y \not =x} c(x,y)$. Given a subset $A \subset \aleph$, we write $\t_A$ for the hitting time to $A$, i.e. $\t_A:= \inf \{ t \geq 0 \,:\, W_t \in A\}$. We denote by $\bbP_x$ the law of the Markov chain starting at $x$ and by $\bbE_x$ the associated expectation. We fix disjoint subsets $A,D \subset \aleph$. Then,  given $\l \in \bbR$,  we consider the functions 
$$ h^{(\l)} _A (x) := \bbE_x ( e^{\l \t_A})\,, \qquad h^{(\l)} _{A,D} (x) := \bbE_x ( e^{\l \t_A} \mathds{1}( \t_A < \t_D) )\,.
$$
Then,
 by conditioning on the first jump time,  one easily obtains that $h_A^{(\l)}$ and $h_{A,D}^{(\l)} $ satisfy the following linear systems:
\begin{equation}
\begin{cases}
(\cL +\l )  h_A^{(\l)}(x)  =0 & x \not \in A\\
h_A^{(\l)} (x) =1 & x \in A
\end{cases}
\qquad \qquad ,\qquad \qquad
\begin{cases}
(\cL +\l )  h_{A,D}^{(\l)}(x)  =0 & x \not \in A\cup D\\
h_{A,D}^{(\l)} (x) =1 & x \in A\\
h_{A,D}^{(\l)} (x) =0 & x \in D
\end{cases}
\,.
\end{equation} Given $n=0,1 , \dots $ let  us now define
\[
w_A^{(n)} (x):= \bbE_x (\t_A ^n) \,, \qquad  w_{A,D}^{(n)} (x):= \bbE_x (\t_A ^n \mathds{1}( \t_A < \t_D)) \,.\]
Then, using that $w_A^{(n)} (x)= \frac{d ^n h_A^{(\l)} }{d\l ^n} (x, \l=0)$ and 
 $w_{A,D}^{(n)} (x) = \frac{d ^n h_{A,D}^{(\l)} }{d\l ^n} (x, \l=0)$, from the above systems one easily derives linear systems satisfied by $w_A^{(n)}$ and $w_{A,D}^{(n)}$. For example one gets
 \begin{equation}\label{orvieto1}
 \begin{cases}
 \cL  w_A^{(0)} (x) 
 =0 & x \not \in A\\
w_A^{(0)} (x) =1 & x \in A
\end{cases}\;,\;
\begin{cases}
 \cL  w_A^{(1)} (x)+1 
 =0 & x \not \in A\\
w_A^{(1)} (x) =0 & x \in A
\end{cases}\;,\;
\begin{cases}
 \cL  w_A^{(2)} (x)+2 w_A^{(1)}(x)
 =0 & x \not \in A\\
w_A^{(2)} (x) =0 & x \in A
\end{cases}
\end{equation}
(the first two systems  correspond to the ones in Theorems 3.3.1 and 3.3.3 of \cite{N}) and 
\begin{equation}\label{orvieto2}
\begin{cases}
 \cL  w_{A,D}^{(0)} (x)=0  & x \not \in A\cup D\\
 w_{A,D}^{(0)} (x)=1 & x \in A\\
w_{A,D}^{(0)} (x) =0 & x \in D
\end{cases}
\qquad , \qquad
\begin{cases}
 \cL  w_{A,D}^{(1)} (x)+ w_{A,D}^{(0)}(x)=0  & x \not \in A\cup D\\
 w_{A,D}^{(1)} (x)=0 & x \in A\cup D\,.
\end{cases}
\end{equation}

 Working with  the random walk $\tilde X$ introduced in Section \ref{sanremo}, to compute $v$ in \eqref{trenino} one can proceed as follows:  $\bbP( X_S= \pm 1_*)$ can be computed by the first system in \eqref{orvieto1} with $A= \{\pm 1_*\}$, $\bbE(S)$ can be computed with the second system in \eqref{orvieto1} with $A= \{-1_*, 1_*\}$. To compute $\s^2$ in \eqref{diffondo1} one can 
 proceed as follows. We have already treated $\bbE(S)$ (the denominator in \eqref{diffondo1}).  For the variance we write 
 \begin{equation}\label{treno}
 \begin{cases}
  {\rm Var} (X^*_S-v S)= {\rm Var} ( X_S^*)+ v^2 ( \bbE(S^2)- \bbE(S)^2) - 2 v {\rm Cov}( X_S^*, S)\,,\\
   {\rm Cov}( X_S^*, S)= \bbE(S \mathds{1}( X_S=1_*))- \bbE(S \mathds{1}( X_S=-1_*))- \bbE(X_{S_*}) \bbE(S)
  \end{cases}
 \end{equation}
 The term  $\bbE(S^2)$ can be computed by means of the third system in \eqref{orvieto1}. The variable   $X_S^*$ takes value $\pm 1$ with probability  $\bbP( X_S= \pm 1_*)$, that we have already computed, thus leading to ${\rm Var} ( X_S^*)$ and $\bbE(X_{S_*})$.  Finally, to compute $ \bbE(S \mathds{1}( X_S=\pm 1_*))$ we use the second system in \eqref{orvieto2} with $A= \{ \pm 1_*\}$ and $D= \{- 1_*\}$.

 \subsection{Reduction to a single cell}
 We now  explain how one can obtain all the interesting quantities, such as asymptotic velocity, diffusion coefficient and rate function,  
by only working with a fundamental cell. Again we restrict to Markovian random walks.
The methods outlined below will be applied in Sections \ref{esempi1} and \ref{esempi3} when considering specific examples. Trivially, the removal of linear chains discussed in the previous subsection is completely general, and can be applied also in the case of a random walk in a single cell with absorbing states. 

 \subsubsection{Asymptotic velocity}
We call $J_k$'s  the consecutive  times at which the Markov chain $\tilde X$  defined in Section \ref{sanremo} hits  the $*$--states:
 \begin{equation}\label{jumping}
	\begin{cases}
	J_0 := 0 \\
	J_k := \inf \{ t \geq J_{k-1}\,:\, X_t \in \{ -1_* , 0_* , 1_* \}\,,\; \exists s \in (J_{k-1},t)  \text{ with } X_s \not = X_{J_{k-1} } 
  \}  \quad k\geq 1 \, . 
	\end{cases} 
	\end{equation}

\begin{Lemma}\label{calcolo_v} Setting
$
\tilde{p} = \bbP (X_{J_1} = 1_* )$ and $ \tilde{q} = \bbP (X_{J_1} = -1_* )$, it holds
\begin{equation}\label{alba_bella1}
\bbP ( X_S = 1_* )=\frac{\tilde{p}}{\tilde{p} + \tilde{q}}\,,\qquad  \bbP ( X_S = -1_* )=\frac{\tilde{q}}{\tilde{p} + \tilde{q}}\,,\qquad \bbE(S) = \frac{ \bbE (J_1 )}{\tilde{p} + \tilde{q}}\,.
\end{equation}
In particular, 
the asymptotic velocity in $v$ in \eqref{trenino} can be written as 
$v= \frac{\tilde{p}- \tilde{q} }{\bbE(J_1)}$.
	 \end{Lemma}
	 The proof is given in Section \ref{montagnola}.
	 
	 \medskip

	The above relations show that, in order to compute the asymptotic velocity $v$ in \eqref{trenino},  it suffices to compute $\bbP ( X_{J_{1}} = \pm 1_* )$, $\bbE ( J_{1} )$. 
This can be done working with just one cell. To this end we let $\hat{X}_t := X_{J_1 \wedge t}$ and define $\bbP_{v_n} ( \cdot ) := \bbP (\cdot | X_0 = v_n )$ for $v \in V$ and $n\in \bbN$, with the convention that $\bbP(\cdot ) =  \bbP_{0_*} ( \cdot ) $.
Then
	\[ \tilde{p} = \bbP (X_{J_1} = 1_* ) = \sum_{ u \in V : (\underline{v} , u) \in E} \bbP (X_{j_1} = u_0 ) \bbP_{u_0} ( \hat{X}_t = 1_*  \mbox{ for some } t\in \bbR_+ ) \]
	where $j_1$ denotes the first jump time of the rw $X$, i.e.\ 
	$ j_1 := \inf\{t \in \bbR_+\,:\, X_t \not = 0_*\}$. We recall that $u_0$ is the site in the $0$--labelled cell of the quasi 1D lattice, corresponding to $u$. 
Again,  the hitting probability $  \bbP_{u_0} ( \hat{X}_t = 1_*  \mbox{ for some } t\in \bbR_+ )$ can be computed by solving the $|V| \times |V|$ linear system in Theorem 3.3.1 of \cite{N}. 
The same method can be applied to obtain $ \tilde{q}=\bbP (X_{J_1} = -1_* )$.

Moreover we can write
	\[ \bbE (J_1 ) = \bbE (j_1 ) + \sum_{ u \in V : (\underline{v} , u) \in E} \bbP (X_{j_1} = u_0 ) \bbE_{u_0} (J_1) 
	+ \sum_{ w \in V : (w, \overline{v}) \in E} \bbP (X_{j_1} = w_{-1} ) \bbE_{w_{-1}} (J_1) \, . \]
Note that $\bbE_{w_{-1}} (J_1) = \bbE_{w_{0}} (J_1)$
 and, being $J_1$ a hitting time for the process $\hat{X}$, $\bbE_{u_{0}} (J_1)$ and $\bbE_{w_{0}} (J_1)$ can be simultaneously computed by solving the $|V| \times |V|$ linear system in Theorem 3.3.3 of \cite{N}.
 
 \subsubsection{Diffusion coefficient}   
 \begin{Lemma}\label{calcolo_diff_coeff} The diffusion coefficient $\s^2$ in Theorem \ref{zecchino} can be written as
 \begin{equation}\label{crema}
 	\s^2  = 
	 \frac{\tilde{p} + \tilde{q}}{\bbE(J_1)} + v^2 \frac{ \bbE(J_1^2)}{\bbE(J_1)} 
	-2v \left( \frac{\bbE ( J_1 \mathds{1} ( X_{J_1} = 1_* ) ) - \bbE ( J_1 \mathds{1} ( X_{J_1} = -1_* ) ) }{\bbE(J_1)} \right) \, ,
 	\end{equation}
 where $\tilde{p} $ and $\tilde{q}$ are defined as in Lemma \ref{calcolo_v}.
 \end{Lemma}
 The proof is given in Section \ref{montagnola}.  We have already explained how to compute the 
 above quantities apart from $\bbE(J_1^2), \bbE ( J_1 \mathds{1} ( X_{J_1} =\pm 1_* ) )$. 
 The idea is to consider the system after the first jump from $0_*$ and then using  the third system in \eqref{orvieto1} and the second system in \eqref{orvieto2} for a suitable random walk inside a single cell. In order to avoid heavy notation, we refer the reader to Sections \ref{esempi1} and \ref{esempi3} for examples.


\section{Examples: $N$--periodic linear model } \label{esempi1}

We now consider a continuous time nearest--neighbor random walk on  $\bbZ$ with $N$--periodic rates, where $N$ is any fixed positive integer, and exponential holding times.  The fundamental graph  $G=(V,E)$ is then given by $V=\{0,1, \dots, N\}$, $E=\{ (x,x+1) \,:\, 0 \leq x <N\} \cup \{ (x,x-1)\,:\, 0<x \leq N\}$.  To simplify the notation, 
we let  
\begin{equation*}
\xi_x^+ :=r(x,x+1)\,, \qquad \xi_x^- := r(x,x-1) \,, \qquad \rho_x =\xi_x^-/\xi_x^+\,, \qquad x \in \bbZ \,.
\end{equation*} 
Note that  $\xi_{x+N}^\pm = \xi_x^\pm$  and $\rho_{x+N}=\rho_x$.
 The LLN and CLT   have been derived  by Derrida in  \cite{D} on the basis of some Ansatz, while a more rigorous derivation of the LLN has been provided in \cite{TF}.

\begin{Proposition}\label{calzino} Set $\D= \rho_1 \cdots \rho_{N}$. Given $1\leq k \leq N$ let  
\begin{align}
r_k& := 
\frac{1}{\xi_k^+} \Big( 1+ \rho_{k+1} + \rho_{k+1} \rho_{k+2}+ \cdots + \rho_{k+1} \rho_{k+2} \cdots \rho_{k+N-1}\Big)\label{ronf} \\
 \L_k& := 1+ \rho_1+ \rho_1 \rho_2 + \cdots +\rho_1 \rho_2\cdots \rho_{k-1}\,,\\
 \Upsilon_k&:= \sum_{i=1}^{k-1} \sum_{n=1}^i \frac{1}{\xi_n^+} \Big( \prod_{j=n+1}^i \rho_j\Big)\,,
\end{align}
while, given $1\leq k \leq N-1$, let
\begin{equation}
  w_k:=\frac{1}{\xi_k^+}  \Big( \rho_{k+1} \cdots \rho_{N} + 
 \rho_{k+1} \cdots \rho_N \rho_{N+1}+   \cdots+ \rho_{k+1} \cdots \rho_{N+k-1}\Big)\,.\label{w100}
\end{equation}
Finally, given $1\leq k \leq N-1$,  set $h_k: = \L_k/\L_N$ and $k_k :=  \Lambda_{k} \Upsilon_{N}/\Lambda_{N} - \Upsilon_{k} $.
Then the asymptotic velocity $v$ and diffusion constant $\s^2$ of the skeleton process $X_*$ are given by
\begin{align}
v & =  (1-\D) / \Big(\sum _{k=1} ^N r_k\Big)\,, \label{razzo}\\
\s^2 & = 
	\frac{1}{ \sum_{k=1}^N r_k } \bigg[ 
	( 1+ \D) + v^2 \Big( \sum_{k=1}^{N-1} 2k_k r_k \Big) 
	+2v \Big( \sum_{k=1}^{N-1} (w_k-h_k r_k)\Big) 
	\bigg] \, . 
	\label{leopardi}
\end{align}
\end{Proposition}
Above, and in what follows, we use  the convention that the  sum of an empty set of  addenda is zero and the product of an empty set of factors is one. In particular, $\L_1= 1$, $\Upsilon_1=0$ and  $\prod_{j=i+1}^i \rho_j=1$.

Note that $v$ coincides with the expression obtained in \cite{D}. Moreover,  when $N=1$,  $X=(X_t)_{t\in\bbR_+}$ is the spatially homogeneous,   continuous time nearest neighbor random walk 
 on $\bbZ$ with rates $r(x,x+1)=\xi_x^+=:\alpha$, $r(x,x-1)=\xi_x^-=:\b$, where $\a,\b>0$.    The fundamental graph $G=(V,E)$  can  be defined  as $V= \{0,1\}$  and $E= \{ (0,1), (1,0) \}$,  and  the marked 
vertices can be taken as $\underline{v}=0, \overline{v}=1$. Note that in this case  the skeleton process $X^*$ coincides with the random walk,    the random time $S$ defined in \eqref{pietro} is an exponential variable of mean $1/(\a+\b)$, independent from $X_S$.
By \eqref{trenino} the asymptotic velocity of the process $X$ is $v = \a-\b $.
Moreover, 
the asymptotic diffusion coefficient is given by 
	\[ \s^2 = \frac{ {\rm Var} (X_S -v S)}{ \bbE (S) } = \frac{ {\rm Var} (X_S -(\a-\b) S)}{ \bbE (S) } 
	=(\a+\b)\left[ {\rm Var} (X_S) + (\a-\b)^2 {\rm Var} ( S)\right] =\a+\b   \, , \]
	since $ {\rm Var} (X_S)= 4\a\b /(\a+\b)^2$ and $ {\rm Var} ( S)= 1/(\a+\b)^2$.
	\medskip
	
	The rest of the section is devoted to the proof of Proposition \ref{calzino}.

\subsection{Asymptotic velocity}\label{mezzaluna}
The computation of the asymptotic velocity $v$  is obtained below  by using  two different methods. The first one is based on the removal of linear chains, as outlined in Section \ref{rimozione}. The second method   uses the reduction to a single cell as in Section \ref{1prigione} (one could as well reduce the problem to a 2--cells graph as in Section \ref{sanremo}). While being longer,  it will allow us to introduce some schemes useful also for the computation of the diffusion coefficient $\s^2$  of the limiting rescaled skeleton process.
%
%
\subsubsection{First method}
 To compute the asymptotic velocity we  remove linear chains as described in Section 
\ref{rimozione}. Note that the points $-N,0,N$ correspond to $-1_*,0_*, 1_*$, respectively.
 We consider the  family of linear chains $\cC = \{ \g , \g '\}$, where 
$$  \g =(0,1, \dots, N)  \,, \qquad  \g' =( -N,-N+1, \dots, 0) \, . 
$$
The resulting graph $\bar G$  defined in Section \ref{rimozione}  has vertices $ -N,0,N$ and edges $(-N,0)$, $(0,-N)$,  $ (0,N)$,  $ (N,0)$. Moreover, the weights $\bar r$ (cf. \eqref{rocco_hunt_1}, \eqref{rocco_hunt_2}) are given by
$$\begin{cases}
& \bar r (-N,0)= \bar r (0,N)= \xi_0^+ /\G \bigl( \g  \bigr) \,, \\
& \bar r (0,-N)= \bar r (N,0)= \xi_N^- /\G \bigl( \bar \g \bigr) = \xi_0^-/ \G \bigl( \bar \g \bigr)\,.
\end{cases}
$$
Recall (cf. \eqref{alexian} and \eqref{supersonico}) that
 \begin{equation} \label{sole}
 \G (\g)
  = 1+\sum_{i=1}^{N-1} \rho_1 \rho_2 \cdots \rho_i \,, \qquad 
 \G( \bar \g)   =\rho_0 \frac{\G(\g)}{\D} \,,\qquad
 \D=\rho_0 \rho_1 \cdots \rho_{N-1}\,.
   \end{equation}

The system \eqref{febbre1bis} trivially implies that
\begin{align}
& \bbP( X_S^*=N)= \frac{ \xi_0^+/ \G \bigl( \g  \bigr)}{  \xi_0^- /\G \bigl( \bar \g \bigr)+ \xi_0^+ /\G \bigl( \g  \bigr)} 
= \frac{1}{1+\Delta} \, , \label{salutare1}  \\
& \bbP( X_S^*= -N)=\frac{ \xi_0^-/ \G \bigl( \bar \g  \bigr)}{  \xi_0^- /\G \bigl( \bar \g \bigr)+ \xi_0^+/ \G \bigl( \g  \bigr)} = 
\frac{\Delta}{1+\Delta} \,.  \label{salutare11}\end{align}
Recall  (cf. \eqref{rachele}) that 
\begin{equation}\label{lego_movie}
  c(\g)  =  \sum_{1\leq k \leq i \leq N-1} \frac{1}{\xi_k^+} \rho_{k+1} \cdots \rho_i
  \,,\; \qquad 
   c(\bar \g)   =\sum_{1\leq k \leq i \leq N-1} \frac{1}{\xi_{-k}^-}
   \rho_{-(k+1)}^{-1} \cdots   \rho_{-i}^{-1} \,.
\end{equation}
 Then the constant $c(0)$ of Proposition \ref{virulino2} is given by
$c(0)= \xi_0^+ c(\g)/\G(\g) + \xi _0^- c(\bar \g)/\G(\bar \g)$.
In particular,  the
  system \eqref{febbre2bis} implies that 
\begin{equation}\label{salutare2}
 \bbE (S)= \psi (0)= \frac{1+  c(0)}{\bar r (0,N)+\bar r (0,-N)}= \frac{ 1+ \xi_0^+ c(\g)/\G(\g) + \xi _0^- c(\bar \g)/\G(\bar \g) }{  \xi_0^+/ \G \bigl( \g  \bigr) +  \xi_0^-/ \G \bigl( \bar \g \bigr)}\,.
 \end{equation}
 Combining \eqref{salutare1}, \eqref{salutare11} and \eqref{salutare2} we get that the asymptotic velocity $v$  of the skeleton process $X^*$ (cf. \eqref{trenino}) is given by
$\bigl[\xi_0^+ /\G \bigl( \g  \bigr) - \xi_0^-/ \G \bigl( \bar \g \bigr)    \bigr] /\bigl[1+ \xi_0^+ c(\g)/\G(\g) + \xi _0^- c(\bar \g)/\G(\bar \g) \bigr]$.
By using the relation  $ \D \xi_0^+ / \G(\g)= \xi_0^- / \G(\bar \g)$,
 the above expression can be simplified to
 \begin{equation}\label{scorpion100}
 v=\frac{1-\D   }{\frac{\G(\g)}{\xi_0^+}  + c(\g)+ \D c( \bar \g)}
 \,.
 \end{equation}
The asymptotic velocity of the rw $X$ is therefore given by $Nv$. 
It is simple to  check that our expression for $v$ is equivalent to  the ones  obtained in \cite{D}[Eq. (49)] and  \cite{TF}[Eq. (4.10)]. We restrict to  \cite{D}  where, under suitable Ansatz,  Derrida  derives   for the asymptotic velocity the value $(1-\D)/\bigl(\sum _{k=1} ^N r_k\bigr)$.
 To recover Derrida's result  we state a technical fact in full generality also for future applications:
 \begin{Lemma} \label{7nani}
Recall the definition of $w_k$  given in  \eqref{w100} and, for $1 \leq k \leq N-1$, set
\begin{equation}
 z_k := \frac{1}{\xi_k^+} \Big( 1+ \rho_{k+1} + \rho_{k+1} \rho_{k+2}+ \cdots + \rho_{k+1} \rho_{k+2} \cdots \rho_{N-1}\Big)  \label{z100}\,.
 \end{equation}
 Let $(x_i)_{i\in \bbZ}$ be a $N$-periodic sequence. Then it holds 
    \[ \sum_{i=1}^{N-1} \sum_{k=1}^i \frac{x_k}{\xi_k^+} \r_{k+1} \cdots \r_i = 
    \sum_{i=1}^{N-1} x_k z_k \, , \qquad 
     \D \sum_{i=1}^{N-1} \sum_{k=1}^i \frac{x_{-k}}{\xi_{-k}^-} \r_{-(k+1)}^{-1} 
     \cdots \r_{-i}^{-1} = 
    \sum_{i=1}^{N-1} x_k w_k \,.\]
     \end{Lemma}
\begin{proof}
By definition, 
	\[ \sum_{i=1}^{N-1} \sum_{k=1}^i \frac{x_k}{\xi_k^+} \r_{k+1} \cdots \r_i = 
	\sum_{k=1}^{N-1} x_k \sum_{i=k}^{N-1} \frac{1}{\xi_k^+} \r_{k+1} \cdots \r_i =
    \sum_{i=1}^{N-1} x_k z_k \, . \]
Moreover, 
we can write (see below for explanations)
\begin{equation*}
\begin{split} \D \sum_{i=1}^{N-1} & \sum_{k=1}^i \frac{x_{-k}}{\xi_{-k}^-} \r_{-(k+1)}^{-1} 
     \cdots \r_{-i}^{-1} 
     = \D \sum _{1\leq u \leq v \leq N-1} \frac{x_v}{\xi_v^-} \rho_u^{-1} \rho_{u+1}^{-1} \dots \rho_{v-1}^{-1}\\
& =
\sum _{1\leq u \leq v \leq N-1}  \frac{x_v}{\xi_v^-} (\rho_0 \rho_1 \cdots \rho_{u-1}) ( \rho_v \rho_{v+1} \cdots \rho_{N -1})
 = \sum_{1 \leq u \leq v \leq N-1} \frac{x_v}{\xi_v^-} \rho_v \rho_{v+1} \cdots \rho_{N+u-1}\\
& =
\sum_{v=1}^{N-1} \frac{x_v}{\xi_v^+}  \Big( \rho_{v+1} \cdots \rho_{N} + 
 \rho_{v+1} \cdots \rho_N \rho_{N+1}+   \cdots+ \rho_{v+1} \cdots \rho_{N+v-1}\Big)= \sum_{v=1}^{N-1} x_v w_v \, . 
\end{split}
\end{equation*}
The first equality follows by translating the indexes by $N$ and calling $v=-k+N, u=-i+N$,  the third equality uses that $\rho_0 \rho_1 \cdots \rho_{u-1}= \rho_N \rho_{N+1} \cdots \rho_{N+u-1}$.
\end{proof}
 
 As a consequence we obtain:
\begin{Corollary}\label{derry} It holds $\sum_{k=1}^N r_k = \frac{\G(\g)}{\xi_0^+}  + c(\g)+ \D c( \bar \g)$. Hence, $v =  \frac{1-\D}{ \sum _{k=1} ^N r_k}$ as in \cite{D}.\end{Corollary}
\begin{proof}
We can write $c(\g)= \sum _{k=1}^{N-1} z_k$. Applying Lemma \ref{7nani} with $x_i\equiv 1 $ we get that $\D c(\bar \g)=\sum _{k=1}^{N-1}w_k$.
The thesis then easily follows by observing that $r_k=z_k+w_k $ for $k=1, \dots, N-1$ and 
that $r_N=r_0 = \G(\g)/ \xi_0^+$.
\end{proof}
%
%
\subsubsection{Second method} Recall the notation introduced at the beginning of Section \ref{1prigione}. By Lemma \ref{calcolo_v}  we have
	$v = ( \tilde{p} - \tilde{q})/\bbE (J_1)$. 
Applying the Markov property at the time of first jump, we get  
$\tilde{p} =  \bbP (X_{J_1} = N) = \frac{\xi_0^+}{\xi_0^- + \xi_0^+} \bbP_1 (X_{J_1} = N) $.  Let $h_i := \bbP_i ( X_{J_1} =  N )$ for $i =  0  \ldots ,N$. Then the vector $( h_0  \ldots , h_N)$ is the minimal non--negative solution of the linear system 
	\begin{equation} \label{sabato}
	 \begin{cases}
	h_{0} = 0\,,\; h_N=1\,, \\
	h_i = \frac{\xi_i^-}{\xi_i^- + \xi_i^+} h_{i-1} + \frac{\xi_i^+}{\xi_i^- + \xi_i^+} h_{i+1} \, , \quad  1 \leq i \leq N-1
	\end{cases} 
	\end{equation}
(see \cite{N}, Theorem 3.3.1).	

\smallskip

In the computations below we shall encounter this type of systems several times, so the following result will reveal handy. To simplify the notation, let $\rho_k := \xi_k^- / \xi_k^+$ for all $k \in \bbZ$. 

\begin{Lemma} \label{lemma_calcoli}
Let $(x_0 , x_1  \ldots x_N)$ solve the linear system 
	\begin{equation}\label{rose} 
	\begin{cases}
	x_0 = 0\,, \;\;  
	x_N = a\,, \\
	x_{i+1} - x_i = \rho_i (x_i - x_{i-1} ) - \a_i \, , \quad i=1, \dots, N-1
	\end{cases} 
	\end{equation}
where $a $ and $\a_1 \ldots \a_N$ are given numbers. Then   $(x_0 , x_1  \ldots x_N)$  is univocally determined:
	\begin{equation*}
	x_1 = \frac{ a + \Upsilon^*_{N}}{\Lambda_{N}} \, , \qquad 
	x_k = \Lambda_{k} x_1 - \Upsilon^*_{k} = \frac{ \Lambda_{k}(a + \Upsilon^*_{N})}{\Lambda_{N}} - \Upsilon^*_{k} \, , 
	\end{equation*}
where we have defined for $k=1, \dots, N$
\begin{equation}\label{lambada}
	\begin{cases}
		\Lambda_k= 1 + \sum_{i=1}^{k-1} \left( \prod_{j=1}^i \rho_j \right)  \,,\\
		\Upsilon^*_k= \sum_{i=1}^{k-1} \sum_{n=1}^i \a_n \left( \prod_{j=n+1}^i \rho_j\right)\,. 
			\end{cases}  
	\end{equation}
\end{Lemma}
 \begin{proof}
Iterating the last equation in \eqref{rose} gives for $i=1,\dots, N-1$
    \[x_{i+1} -x_i = \rho_i \rho_{i-1} \cdots \rho_{1}(x_{1} - x_{0}) -\sum_{n=1}^{i}\a_n \prod_{j=n+1}^{i} \rho_j =
    x_{1} \Big( \prod_{j=1}^i \rho_j  \Big)  - \sum_{n=1}^{i}\a_n \Big(\prod_{j=n+1}^{i} \rho_j \Big)\,, \]
    (with the convention that $  \prod_{j=i+1}^{i} \rho_j =1$). 
Now use the boundary conditions to get:
    \[
    x_N-x_{0}  = a =
     \sum_{i=0}^{N-1} (x_{i+1}-x_i) 
     = x_{1} \Big( 1+ \sum_{i=1}^{N-1} \prod_{j=1}^i \rho_j \Big) - \sum_{i=1}^{N-1} \sum_{n=1}^{i}\a_n \Big(\prod_{j=n+1}^{i} \rho_j \Big)
     = x_1 \Lambda_{N} -\Upsilon^*_{N} 
     \]
which gives $x_1 = (a +\Upsilon^*_{N})/\Lambda_{N}$ as claimed.
Finally, for $k =2,\dots,  N-1 $ we have
	\[ 
	x_k = x_0 + \sum_{i=0}^{k-1} (x_{i+1}-x_i) = 
	x_1\Big( 1+ \sum_{i=1}^{k-1} \prod_{j=1}^i \rho_j \Big) - \sum_{i=1}^{k-1} \sum_{n=1}^{i}\a_n \Big(\prod_{j=n+1}^{i} \rho_j \Big)
     = x_1 \Lambda_{k} - \Upsilon^*_{k}  \, , 
	\]
and this concludes the proof.
\end{proof}

Going back to the computation of $h_i = \bbP_i ( X_{J_1} =N)$, we observe that the last equation
in \eqref{sabato} is equivalent to $h_{i+1} -h_i = \rho_i (h_i - h_{i-1})$. 
Then by  Lemma \ref{lemma_calcoli} with $a=1$ and $\a_1 = \cdots = \a_{N-1}=0$, we  get that $h_i$ are as in Proposition \ref{calzino}, in particular $h_1= 1/\L_N$ and $h_{N-1}=  \Lambda_{N-1}/\Lambda_{N}$. 
From this we deduce  
	\begin{equation}\label{tindari}
	 \begin{cases}
	\tilde{p} =  \frac{\xi_0^+}{\xi_0^- + \xi_0^+} h_1 = \frac{\xi_0^+}{\xi_0^- + \xi_0^+}  \frac{1}{\Lambda_{N}} \\
	\tilde{q} = \frac{\xi_0^-}{\xi_0^- + \xi_0^+} (1-h_{N-1})  = \frac{\xi_0^-}{\xi_0^- + \xi_0^+} \bigg( \frac{ \Lambda_{N} - \Lambda_{N-1}}{\Lambda_{N}}\bigg)  = \frac{\xi_0^-}{\xi_0^- + \xi_0^+} \frac{ \D}{\rho_0\Lambda_{N}} = \D \tilde{p} 
	\end{cases}
	\end{equation}
where $\D=  \rho_1 \dots \rho_N$ as defined in \eqref{sole}.
Above we have used    that $\tilde q$ equals the probability, starting at $N$, to jump to $N-1$ and afterwards  to reach $0$ before $N-1$. 

\smallskip

It remains to compute $\bbE (J_1)$. Conditioning  to the first jump (taking place after an exponential time of mean $(\xi_0^- + \xi_0^+)^{-1}$), we get	\begin{equation} \label{pistacchio}
	\bbE (J_1) = \frac{1}{\xi_0^- + \xi_0^+}+
	\frac{\xi_0^+}{\xi_0^- + \xi_0^+} \bbE_1 (J_1) + \frac{\xi_0^-}{\xi_0^- + \xi_0^+} \bbE_{-1} (J_1) \, . 
	\end{equation}
Above $\bbE_i$ denotes the expectation for the random walk $X$ starting at $i$. Define $k_i=\bbE_i( \mbox{hitting time of } \{0,N\}) $. Note that $k_1= \bbE_1 (J_1) $ and $k_{N-1}=E_{-1} (J_1)$.
Then (cf. Theorem 3.3.3 of \cite{N}) $( k_0 \ldots  k_N)$ is the minimal non--negative solution of
    \begin{equation}\label{menta}
    \begin{cases}
    k_{0}=0 \,,\;\;    k_N=0\,,\\
        k_i=\frac{1}{ \xi_i^-+ \xi_i^+ } + \frac{\xi_i^+}{ \xi_i^- + \xi_i^+ } k_{i+1} + \frac{\xi_i^-}{ \xi_i^- + \xi_i^+ } k_{i-1}\,, \qquad
    1 \leq i \leq N-1 \, . 
    \end{cases}
    \end{equation}
Being the last equation equivalent to
    $k_{i+1} -k_i = \rho_i (k_i - k_{i-1}) -\frac{1}{\xi_i^+}$, we again apply Lemma \ref{lemma_calcoli}, now with $a=0$, $\a_i = 1/\xi_i^+$. In particular, $k_i$ coincides with the definition given in Proposition \ref{calzino}, moreover    $k_{1}   = \Upsilon_N/\Lambda_{N}$ where 
      \[  \Upsilon_{ N} := \sum_{i=1 }^{N-1} \sum_{n=1 }^i \frac{1}{\xi_n^+} \prod_{j=n+1}^i \rho_j
  = \sum_{1\leq n \leq i \leq N-1}  \frac{1}{\xi_n^+} \prod_{j=n+1}^i \rho_j\,.
  \]
    Note that $ \L_N= \G(\g)$ and $  \Upsilon_{ N}=c(\g)$, with $\G(\g)$, $c(\g)$ defined in \eqref{sole} and \eqref{lego_movie} respectively.  In conclusion, $k_1= \bbE_1(J_1)= c(\g)/\G(\g)$.
    Instead of expressing $k_{N-1}= \bbE_{-1} (J_1)$  by means of  Lemma \ref{lemma_calcoli}, we simply invoke symmetry (a  space inversion w.r.t. the origin)   to conclude that $k_{N-1}= c(\bar \g)/\G(\bar \g)$ (recall the notation of \eqref{sole} and \eqref{lego_movie}). Coming back to \eqref{pistacchio} and   using the relation $ \D \xi_0^+ / \G(\g)= \xi_0^- / \G(\bar \g)$
 we obtain
        \begin{equation*}
    \bbE (J_1)    = \frac{\xi_0^+}{\xi_0^- + \xi_0^+}   
    \left[ \frac{1}{\xi_0^+} +\frac{c(\g)}{\G(\g)}+\frac{\xi_0^-}{\xi_0^+}  \frac{c(\bar \g)}{\G(\bar \g)} 
  \right]= \frac{\xi_0^+}{\xi_0^- + \xi_0^+}   \frac{1}{ \G(\g)}
    \left[ \frac{\G(\g) }{\xi_0^+} +c(\g)+\D c(\bar \g)
  \right] \,.
  \end{equation*}
     Recalling that $v = ( \tilde{p} - \tilde{q})/\bbE (J_1)$ and using \eqref{tindari}, we recover \eqref{scorpion100} and therefore \eqref{razzo} by Corollary \ref{derry}.
     
     For later use we point out that the above equation for $\bbE(J_1)$, the identity $\G(\g)= \L_N$ and Corollary  \ref{derry} 
imply that
\begin{equation}\label{cipolline}
\bbE (J_1)    =\frac{\xi_0^+  }{\xi_0^+ + \xi_0^-} \frac{1}{\L_N}   \big( \sum_{n=1}^N r_n \big)
\,.
\end{equation}
    
    \subsection{Diffusion coefficient}
We   use Lemma \ref{calcolo_diff_coeff}. Due to the computations in the previous  subsection
it remains  to compute $\bbE(J_1^2) $ and $\bbE(J_1 \mathds{1}( X_{J_1}=\pm 1_*)) $.

    Given $i =1,2, \dots, N-1$ call $s_i= \bbE_i(J_1 \mathds{1}( X_{J_1}=\pm 1_*)) $ and recall that $h_i= \bbP_i( X_{J_1}= N)$. We set $s_0=s_N=0$, $h_0=0, h_N=1$. 
          Conditioning on the first jump (taking  place at the time $j_1$)  we get 
    \begin{equation}\label{birra}
    \begin{split}
    \bbE(J_1 \mathds{1}( X_{J_1}= 1_*)) & =    \bbE(j_1 \mathds{1}( X_{J_1}= 1_*)) + 
           \bbE((J_1-j_1) \mathds{1}( X_{J_1}= 1_*))\\
           &  =    \frac{\xi_0^+}{(\xi_0^-+\xi_0^+)^2}  h_1+   \frac{\xi_0^+}{\xi_0^-+\xi_0^+}  s_1\,.
           \end{split}
    \end{equation}
   By the second system in \eqref{orvieto2} with $ A = \{1_*\}$, $D=\{0_*\}$, we have
    \begin{equation} \label{violoncello}
    \begin{cases}
    s_{i+1}-s_{i}= \rho_i (s_{i}-s_{i-1} ) - \frac{ h_i }{\xi_i^+}
  \\
  s_0= s_N=0\,.
  \end{cases}
  \end{equation} 
 Applying  Lemma \ref{lemma_calcoli} we get 
  $s_k= \L_k \Psi _N/\L_N - \Psi _k$, where $\L_k= 1+ \sum _{i=1}^{k-1} \prod _{j=1}^i \rho_j$ and $ \Psi _k = \sum _{i=1}^{k-1} \sum _{n=1}^i \frac{h_n}{\xi_n^+} \prod _{j=n+1}^i \rho_j$. 
  In particular, $s_1 = \Psi_N/\Lambda_N$. 
  Coming back to \eqref{birra}, and recalling that $h_1 = 1/\L_N$ (cf. the discussion before \eqref{tindari}),  we get 
   \begin{equation}\label{wodka}
    \bbE(J_1 \mathds{1}( X_{J_1}= 1_*)) = 
    \frac{\xi_0^+}{( \xi_0^-+\xi_0^+) }   \frac{1}{\Lambda_N} \bigg[ 
           \frac{1}{\xi_0^-+\xi_0^+}  + \Psi_N \bigg] \,.
    \end{equation}
By symmetry, then, it holds 
   \begin{equation}\label{wodka2}
    \bbE(J_1 \mathds{1}( X_{J_1}= -1_*)) = 
    \frac{\xi_0^-}{( \xi_0^-+\xi_0^+) }   \frac{1}{\bar{\Lambda}_N} \bigg[ 
           \frac{1}{\xi_0^-+\xi_0^+}  + \bar{\Psi}_N \bigg] 
	= \frac{\xi_0^+}{( \xi_0^-+\xi_0^+) }   \frac{\D}{\Lambda_N} \bigg[ 
           \frac{1}{\xi_0^-+\xi_0^+}  + \bar{\Psi}_N \bigg] 
           \,
    \end{equation}
where $\bar{\L}_N= 1+ \sum _{i=1}^{N-1} \prod _{j=1}^i \rho_{-j}^{-1}$, $\bar{ \Psi} _N = \sum _{i=1}^{N-1} \sum _{n=1}^i \frac{\bar h_{-n}}{\xi_{-n}^-} \prod _{j=n+1}^i \rho_{-j}^{-1}$ and $\bar h_i= \bbP_i( X_{J_1}= -1_*)$. Note that in the second identity in \eqref{wodka2} we have used
 the relation $\bar{\L}_N = \r_0 \L_N / \D$.  
Since $\bar h_{-n}= 1- h_{-n}$,
 using      Lemma \ref{7nani}  with $x_k=1-h_k$, we  get
	\[ \Psi_N = \sum_{k=1}^{N-1} z_k h_k \, , \qquad
	 \D \bar{\Psi}_N = \sum_{k=1}^{N-1} w_k (1-h_k) \, .\]
Coming back to \eqref{wodka2}, using \eqref{cipolline}  and that $w_k+z_k=r_k$, we get 	\begin{equation} \label{parte1}
	 \frac{ \bbE(J_1 \mathds{1}( X_{J_1}= 1_*)) -\bbE(J_1 \mathds{1}( X_{J_1}= -1_*)) }{\bbE(J_1)} 	=
	\frac{1}{\sum_{n=1}^N r_n} \bigg[ \frac{1-\D}{\xi_0^- + \xi_0^+} + 
	\sum_{k=1}^{N-1} (h_k r_k-w_k) \bigg] \, , 
	\end{equation}
with the $r_n$'s defined as in \eqref{ronf}. \\

Let us now concentrate on $\bbE(J_1^2)$. 
As in the previous subsection we set $k_i = \bbE_i(J_1)$ for $1\leq i \leq N-1$ and $k_0=k_N=0$.
Using  the strong Markov property at the first jump  time $j_1$ (note that $\bbE(j_1)=1/(\xi_0^++\xi_0^-)$, $\bbE(j_1^2)=2/(\xi_0^++\xi_0^-)^2$ ), the periodicity and relation \eqref{pistacchio}, we  get
	\begin{equation} \label{chiave}
	 \begin{split} 
	\bbE(J_1^2) & =
	\bbE ( [j_1 + ( J_1 - j_1)]^2 | X_{j_1} = 1) \frac{\xi_0^+}{\xi_0^+ + \xi_0^-} 
	+ \bbE ( [j_1 + ( J_1 - j_1)]^2 | X_{j_1} = -1) \frac{\xi_0^-}{\xi_0^+ + \xi_0^-} \\
	& = \bbE(j_1^2) + \frac{\xi_0^+}{\xi_0^+ + \xi_0^-}  \bigg( \bbE_1 (J_1^2) + 2 \bbE(j_1) \bbE_1(J_1) \bigg) 
	+ \frac{\xi_0^-}{\xi_0^+ + \xi_0^-}  \bigg( \bbE_{-1} (J_1^2) + 2 \bbE(j_1) \bbE_{-1}(J_1) \bigg) 
	\\ & = \frac{2}{(\xi_0^+ + \xi_0^-)^2} + \frac{\xi_0^+}{\xi_0^+ + \xi_0^-} \bigg( \bbE_1(J_1^2)  + \frac{2k_1}{\xi_0^+ + \xi_0^-} \bigg)
	+ \frac{\xi_0^-}{\xi_0^+ + \xi_0^-} \bigg( \bbE_{N-1}(J_1^2) + \frac{2k_{N-1}}{\xi_0^+ + \xi_0^-} \bigg)
	\\ & = \frac{\xi_0^+ \bbE_1(J_1^2)  + \xi_0^- \bbE_{N-1}(J_1^2) }{\xi_0^+ + \xi_0^-} + \frac{2 \bbE (J_1)}{\xi_0^+ + \xi_0^-} \, ,
	\end{split} 
	\end{equation}
Note that $\bbE_1(J_1)$ has already been computed. It only remains to compute $\bbE_1(J_1^2)$, from which an expression for 	$\bbE_{N-1}(J_1^2)$ will follow by symmetry. 
To this aim, let $\g_i := \bbE_i (J_1^2 )$ for $i=1,2 \ldots N-1$, and set $\g_0 = \g_N = 0$. 
Then, by using the third linear system in \eqref{orvieto1} with $A=\{ 0_* , 1_* \}$, we 
have
	\[ \begin{cases}
	\g_i = \frac{\xi_i^+ \g_{i+1} + \xi_i^- \g_{i-1}}{\xi_i^+ + \xi_i^-} + \frac{2 k_i}{\xi_i^+ + \xi_i^-} \\ 
	\g_0 = \g_N = 0 \, . 
	\end{cases}\]
The second equation is equivalent to $\g_{i+1} - \g_i = \rho_i (\g_i - \g_{i-1} ) -2\frac{k_i}{\xi_i^+} $, so we can use Lemma \ref{lemma_calcoli} to get 
	\begin{equation} \label{conf1}
	\g_1 = \Phi_N / \L_N \, , \qquad 
	\Phi_N := \sum_{i=1}^{N-1} \sum_{n=1}^i \frac{2k_n}{\xi_n^+} \prod_{j=n+1}^{i} \r_j \, . 
	\end{equation}
By symmetry  and since $\bar{\L}_N = \r_0 \L_N / \D$ we then have 
	\begin{equation} \label{conf2}
	\g_{N-1} =  \bar{\Phi}_N / \bar{\L}_N = \frac{\D \bar{\Phi}_N}{\r_0 \L_N} \, , \qquad 
	\bar{\Phi}_N := \sum_{i=1}^{N-1} \sum_{n=1}^i \frac{2k_{-n}}{\xi_{-n}^-} \prod_{j=n+1}^{i} \r_{-j}^{-1} \, . 
	\end{equation}
Finally, we can  apply Lemma \ref{7nani} with $x_i = 2k_i$, to get
	\[ \Phi_N + \D \bar{\Phi}_N = \sum_{i=1}^{N-1} 2k_i ( z_i + w_i ) 
	= \sum_{i=1}^{N-1} 2k_i r_i \, , \]
from which, recalling \eqref{chiave}, we obtain
	\begin{equation} \label{parte2} 
	\begin{split}
	\bbE(J_1^2) & = 	\frac{\xi_0^+  }{\xi_0^+ + \xi_0^-} \frac{1}{\L_N} 
	\Big( \Phi_N + \D\bar \Phi_N \Big) + \frac{2 \bbE (J_1)}{\xi_0^+ + \xi_0^-} \\ &
	= 	\frac{\xi_0^+  }{\xi_0^+ + \xi_0^-} \frac{1}{\L_N}  
	\bigg(  \sum_{n=1}^{N-1} 2k_n r_n \bigg) + \frac{2 \bbE (J_1)}{\xi_0^+ + \xi_0^-} 
	\, .  \end{split} 
	\end{equation}
Let us come back to \eqref{crema}. Since $v= (\tilde p- \tilde q)/ \bbE(J_1)$ we can write (recall \eqref{tindari} and Corollary \ref{derry}) $(\tilde p+ \tilde q )/\bbE(J_1)= v(1+\D)/(1-\D)= (1+\D)/  \sum_{n=1}^N r_n$.
In conclusion, plugging  this last relation, \eqref{parte1} and \eqref{parte2} into \eqref{crema}, and recalling \eqref{cipolline}, we end up with the following expression for the diffusion coefficient in the $N$-periodic model:
\begin{equation*}
	\s^2 = \frac{2v^2}{\xi_0^- + \xi_0^+}  +
	\frac{1}{ \big( \sum_{n=1}^N r_n \big)} \bigg[ 
	( 1+ \D) + v^2 \Big( \sum_{n=1}^{N-1} 2k_n r_n \Big) 
	-2v \Big( \frac{1-\D}{\xi_0^- + \xi_0^+}  + \sum_{n=1}^{N-1} (h_n r_n - w_n ) \Big) 
	\bigg] \, . 
	\end{equation*}
Using now that $v= (1-\D)/ \sum_{n=1}^N r_n $ (cf.  Corollary \ref{derry}) we get \eqref{leopardi}.

 


\section{Examples: Parallel--chain model}\label{esempi3}
Let the fundamental graph $G = (V,E)$ be composed by two parallel chains as  in Figure \ref{parallel_fig} below, and $\cG = (\cV , \cE )$ be the correspondent quasi 1D lattice. Call $(a_i^k)_{i=1}^{N_1-1}$ and $(b_j^k)_{j=1}^{N_2-1}$ the intermediate vertices of the upper and lower chain respectively in the $k$th cell. Wherever no confusion occurs, we shall omit the cell index, so that $a_i = a_i^k$ and $b_j = b_j^k$.  The random walk $X$ has exponential holding times and its jump rates are indicated in Figure \ref{parallel_fig}. In addition, we set $\xi_0^+:= r(0_*, a_1^0)$,  $\xi_0^-:= r(0_*, a_{N_1-1}^{-1})$,  $\eta_0^+:= r(0_*, b_1^0)$,  $\eta_0^-:= r(0_*, b_{N-1}^{-1})$ and set  
   $\n_0 := \xi_0^+ + \xi_0^- + \eta_0^+ + \eta_0^-$ for brevity. 

\begin{figure}[!ht]
    \begin{center}
     \centering
  \mbox{\hbox{
  \includegraphics[width=.7\textwidth]{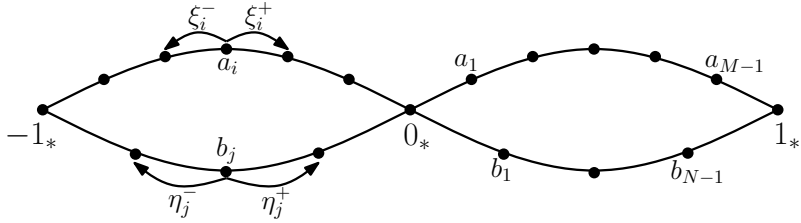}}}
            \end{center}
            \caption{Two cells in the parallel chains model. }\label{parallel_fig}
  \end{figure}

\subsection{Asymptotic velocity}
Call $\g^{(1)}$ and $\g^{(2)}$ the upper and lower chain respectively, and set 	\[ 
	\r^{(1)}_i  := \frac{\xi_i^-}{\xi_i^+}, \qquad  
	\r^{(2)}_j  := \frac{\eta_j^-}{\eta_j^+}, 
	\]
for $0 \leq i \leq N_1 -1$, $0 \leq j \leq N_2 -1$.
Whenever convenient, we think about the jump rates and the $\r_j^{(i)}$'s as to be extended by periodicity, that is: 
$\r^{(i)}_{j \pm N_i} = \r^{(i)}_{j}$, for $i = 1,2$. 
Set 
	\begin{equation}\label{pasqua1} \D^{(i)} := \r^{(i)}_0  \r^{(i)}_1  \cdots \r^{(i)}_{N_i-1} \, , \qquad i=1,2. \end{equation}

Given a linear chain $\g$, recall the definitions of $\G (\g)$ and $c(\g)$ given in \eqref{alexian} and \eqref{rachele}, that 
now read:
	\begin{equation}\label{pasqua2}\begin{split} 
		\G (\g^{(i)}) = 1 + \sum_{k=1}^{N_i -1} \prod_{j=1}^k \r_j^{(i)}  \, , \qquad \quad 
	& c (\g^{(i)}) =  \sum_{1 \leq k \leq j \leq N_i -1} \frac{1}{p_k^{(i)}} \prod_{n=k+1}^j \r_n^{(i)} \, , 
	\end{split} \end{equation}
for $i=1,2$, where for brevity we have set $p_k^{(1)} := \xi_k^+$ and $ p_k^{(2)} := \eta_k^+$. 

We apply the chain removal method presented in Section \ref{rimozione}. Recall the definition of $\tilde{G} = (\tilde{V} , \tilde{E})$ from Section \ref{marina}. In this context $\tilde{G}$ consists of $4$ linear chains, pairwise identical to $\g^{(1)}$ and $\g^{(2)}$. 
We reduce all these chains by defining, according to \eqref{rocco_hunt_1} and \eqref{rocco_hunt_2}, the new jump rates 
	\begin{equation} \label{c2}
	 r_+ := \frac{\xi_0^+}{\G(\g^{(1)})} + \frac{\eta_0^+}{\G(\g^{(2)})} \, , 
	\qquad \quad 
	r_- := \frac{\xi_0^-}{\G(\bar{\g}^{(1)})} + \frac{\eta_0^-}{\G(\bar{\g}^{(2)})}
	=  \frac{\xi_0^+ \D^{(1)} }{\G(\g^{(1)})} + \frac{\eta_0^+ \D^{(2)} }{\G(\g^{(2)})} \, , 
	\end{equation}
where the last equality follows from \eqref{supersonico}. 


 Proposition \ref{virulino1} gives 
$ \bbP( X_S = \pm 1_*)=  \frac{r_\pm }{r_+ + r_-} $, while 
 Proposition \ref{virulino2} gives $\ds \bbE(S) = \frac{1 + c(0_*)}{r_+ + r_-} $, where
	\begin{equation} \label{venerdi}
	 \begin{split} 
	c(0_*)  & = \frac{ \xi_0^+ c (\g^{(1)})}{\G(\g^{(1)})} + \frac{ \eta_0^+ c (\g^{(2)})}{\G(\g^{(2)})}  
	+ \frac{ \xi_0^- c (\bar{\g}^{(1)})}{\G(\bar{\g}^{(1)})} + \frac{ \eta_0^- c (\bar{\g}^{(2)})}{\G(\bar{\g}^{(2)})}  \\
	& = \frac{ \xi_0^+ }{\G(\g^{(1)})} \Big( c (\g^{(1)}) + \D^{(1)} c (\bar{\g} ^{(1)}) \Big) 
	+  \frac{ \eta_0^+ }{\G(\g^{(2)})} \Big( c (\g^{(2)}) + \D^{(2)} c (\bar{\g} ^{(2)}) \Big) \, . 
	\end{split} 
	\end{equation}
	By Theorem \ref{zecchino} it holds $v = \frac{r_+ - r_-}{1 + c(0_*)} $.  The above formula \eqref{venerdi} can be easily simplified. To this aim note that, by periodicity,
	$$
	 c (\bar \g^{(i)}) =  \sum_{1 \leq k \leq j \leq N_i -1} \frac{1}{\bar p_{N_i-k}^{(i)}} \prod_{n=k+1}^j \left(\r_{N_i-n}^{(i)}\right)^{-1}=  \sum_{1 \leq k \leq j \leq N_i -1} \frac{1}{\bar p_{-k}^{(i)}} \prod_{n=k+1}^j \left( \r_{-n}^{(i)}\right)^{-1}
	$$
	where $\bar p^{(1)}_i=\xi_i^-$, $ \bar p^{(2)}_i=\eta_i^-$.
Then, by Lemma \ref{7nani}	we have
	\[ c (\g^{(i)}) + \D^{(i)} c (\bar{\g} ^{(i)})  = \sum_{k=1}^{N_i -1} r_k^{(i)} \, \qquad i=1,2 \]
where, in analogy with \eqref{ronf},
 we have set
	\begin{equation}\label{pasqua3}
	r_k^{(1)} := \frac{1}{\xi_k^+} \Big( 1 +\sum _{i=1}^{N_1-1} \prod _{j=1}^i \rho^{(1)}_{k+j} \Big) \, , 
	\qquad 
	r_k^{(2)} := \frac{1}{\eta_k^+} \Big( 1 +\sum _{i=1}^{N_2-1} \prod _{j=1}^i \rho^{(2)}_{k+j} \Big) \, . \end{equation}

Noting that $\ds r_{N_1}^{(1)} = \frac{ \G (\g^{(1)})}{\xi_0^+}$ and $\ds r_{N_2}^{(2)} = \frac{ \G (\g^{(2)})}{\eta_0^+}$, 
equation \eqref{venerdi} becomes
	\begin{equation} \label{c1}
	 c (0_*) =\frac{1}{r_{N_1} ^{(1)} } \bigg( \sum_{k=1}^{N_1} r_k^{(1)} \bigg )+ \frac{1}{r_{N_2} ^{(2)} } \bigg( \sum_{k=1}^{N_2} r_k^{(2)}\bigg) - 2\, . 
	 \end{equation}
Putting all together,  since  $v = \frac{r_+ - r_-}{1 + c(0_*)} $, we get:
\begin{Proposition}\label{ovetto} The asymptotic velocity of the skeleton process $X_*$ is given by
	\begin{equation}\label{v_parallel_c} v =\frac{
	\frac{1}{r_{N_1} ^{(1)} }   (1-\D^{(1)}) +  \frac{1}{r_{N_2} ^{(2)} } (1-\D^{(2)}) }{  \frac{1}{r_{N_1} ^{(1)} } \sum_{k=1}^{N_1} r_k^{(1)} 
	+  \frac{1}{r_{N_2} ^{(2)} } \sum_{k=1}^{N_2} r_k^{(2)}  -1
	} \, ,
	\end{equation}
	where  $\D^{(i)}$, $\G(\g^{(i)})$ and $r_k^{(i)} $  are  defined in \eqref{pasqua1}, \eqref{pasqua2} and \eqref{pasqua3}, respectively.
	\end{Proposition}
Notice that this formula matches equation	(6) of \cite{K1} (our $N_1$, $N_2$,  $r_{N_1} ^{(1)}$, $r_{N_2} ^{(2)}$, $\D^{(1)}$, $\D^{(2)}$,    $\sum_{k=1}^{N_1} r_k^{(1)} $,  $\sum_{k=1}^{N_2} r_k^{(2)} $     equal    $N$, $M$,
$r_0^{(0)}$, $r_0^{(1)}$, $\P_{(0)1}^{N}$,  $\P_{(1)1}^{M}$, $R_{N}$, $R_{M}$
  in \cite{K1}, respectively).

\subsection{Diffusion coefficient}
By using the reduction to a single cell introduced in Section \ref{1prigione}, together with the computations for the linear chain case, we are able to compute the diffusion coefficient in the parallel chains case. 

Recall that $\s^2$ is given by formula \eqref{crema} in Lemma \ref{calcolo_diff_coeff}.  
  By Lemma \ref{calcolo_v}   and since $\ds \bbE(S) = \frac{1 + c(0_*)}{r_+ + r_-} $ (see previous subsection), we get (recall  \eqref{alba_bella1} and \eqref{c1}):
  \begin{equation}\label{mattone1}
\boxed{\frac{\tilde p + \tilde q}{ \bbE(J_1)}= \frac{1}{\bbE(S)}= \frac{r_++r_-}{  \frac{1}{r_{N_1} ^{(1)} } \sum_{k=1}^{N_1} r_k^{(1)} 
	+  \frac{1}{r_{N_2} ^{(2)} } \sum_{k=1}^{N_2} r_k^{(2)}  -1}\,.}
\end{equation}
To compute $\bbE(J_1)$ we need to compute $\tilde{p}, \tilde{q}$.
To compute $\tilde{p}$, note that on the event $\{ X_{J_1} = 1_* \}$ the process must be either in $a_1$ or in $b_1$ after the first jump. By the strong Markov property, then, we have
	\[ \tilde{p} = \bbP (X_{J_1} = 1_*) = \frac{\xi_0^+}{\n_0} h_1^{(1)} + \frac{\eta_0^+}{\n_0} h_1^{(2)} \, , \]
where $ h_1^{(1)}  = \bbP_{a_1} (X_{J_1} = 1_*) $ and $  h_1^{(2)}  = \bbP_{b_1} (X_{J_1} = 1_*) $, in analogy with the notation introduced in \eqref{sabato}. 

Note that, by definition of the random time $J_1$, after the process $X$ picks a linear chain with the first jump, it does not leave it at least up to time $J_1$.
We can therefore apply the results in Section \ref{esempi1} (see the discussion after the proof of Lemma \ref{lemma_calcoli}) to find
    \begin{equation} \label{altri_h}
 	 h_1^{(1)} = \frac{1}{\G \big( \g^{(1)} \big)} , \qquad 
	h_1^{(2)} = \frac{1}{\G \big( \g^{(2)} \big)}  \, , 
	\end{equation}
where $\G \big( \g^{(i)} \big)$, $i=1,2$ are defined as in \eqref{pasqua2}. 

Hence we have (see \eqref{c2})
	\[ \tilde{p} =  \frac{\xi_0^+}{\n_0} \frac{1}{\G \big( \g^{(1)} \big)}
	+ \frac{\eta_0^+}{\n_0} \frac{1}{\G \big( \g^{(2)} \big)}  = \frac{r_+}{\n_0} \, , \]
and by symmetry
	\[ \tilde{q} =  \frac{\xi_0^-}{\n_0} \frac{1}{\G \big( \bar{\g}^{(1)} \big)}
	+ \frac{\eta_0^-}{\n_0} \frac{1}{\G \big( \bar{\g}^{(2)} \big)}  = \frac{r_-}{\n_0} \, .\]
	Due to the previous expressions for $\tilde p , \tilde q$ and due to \eqref{mattone1}, we get
	\begin{equation}\label{mattone2} 
	\boxed{\bbE (J_1) =  \frac{1}{\n_0} \Bigg(  \frac{1}{r_{N_1} ^{(1)} } \sum_{k=1}^{N_1} r_k^{(1)} 
	+  \frac{1}{r_{N_2} ^{(2)} } \sum_{k=1}^{N_2} r_k^{(2)}  -1
 \Bigg) \, . }
 \end{equation}
 It remains to compute $\bbE(J_1 \mathds{1} ( X_{J_1} = \pm 1_*))$ and $\bbE(J_1^2)$.
To this end, we condition on the position of the process after the first jump, and reduce to the linear chain case, for which all interesting quantities have already been computed.\\

Let us start with the computation of $\bbE(J_1 \mathds{1} ( X_{J_1} =  1_*))$, from which a formula for $\bbE(J_1 \mathds{1} ( X_{J_1} = - 1_*))$ will follow by symmetry.
In analogy with \eqref{birra}, if we let $j_1$ denote the first jump time of the process $X$, then we have
	\begin{equation}\label{vetro} 
	\begin{split}
	\bbE(J_1 \mathds{1} ( X_{J_1} =  1_*) )& = 
	\bbE(j_1 \mathds{1} ( X_{J_1} =  1_*)) + 
	\bbE\bigl((J_1-j_1) \mathds{1} ( X_{J_1} =  1_*,  X_{j_1} = a_1)\bigr ) \\ & + 
	 \bbE\bigl((J_1-j_1) \mathds{1} ( X_{J_1} =  1_*,  X_{j_1} = b_1)\bigr) \\
	 & = \frac{\xi_0^+  }{\n_0^2} \bbP_{a_1} (X_{J_1} = 1_*) + 
	 \frac{\eta_0^+  }{\n_0^2} \bbP_{b_1} (X_{J_1} = 1_*) \\ & + 
	 \frac{\xi_0^+}{\n_0} \bbE_{a_1} (J_1 \mathds{1} (X_{J_1} = 1_*)) +
	 \frac{\eta_0^+}{\n_0} \bbE_{b_1} (J_1 \mathds{1} (X_{J_1} = 1_*)) \\
	 & = \frac{\xi_0^+ h_1^{(1)} + \eta_0^+ h_1^{(2)}}{\n_0^2} 
	 + \frac{\xi_0^+ s_1^{(1)} + \eta_0^+ s_1^{(2)}}{\n_0} \, , 
	\end{split} 
	\end{equation}	 
where, similarly to the periodic linear case, we have called 
$s_1^{(i)}$ the quantities corresponding to 
$s_1$ in Section \ref{esempi1} where the linear chain is replaced by $\g^{(i)}$, $i=1,2$ (see \eqref{violoncello} and the discussion preceding it).


 Note that, using that $h_{k+1}^{(i)} = h_k^{(i)} + \r_k^{(i)}  ( h_{k}^{(i)}  - h_{k-1}^{(i)}  ) $ together with $h_0^{(i)} = 0$ for $i=1,2$, one can obtain $h_{2}^{(i)}  , \ldots , h_{N_i-1}^{(i)} $ from \eqref{altri_h}. 
Moreover, similarly to \eqref{violoncello}, we find $s_1^{(i)} = \Psi^{(i)} / \G \big( \g^{(i)} \big)$, $i=1,2$, where we have set
	\[ \Psi^{(1)} := \sum_{k=1}^{N_1 -1} \sum_{n=1}^k \frac{ h_n^{(1)} }{ \xi_n^+} 
	\prod_{j=n+1}^k \r_j^{(1)} \, , \qquad 
	 \Psi^{(2)} := \sum_{k=1}^{N_2 -1} \sum_{n=1}^k \frac{ h_n^{(2)} }{ \eta_n^+} 
	\prod_{j=n+1}^k \r_j^{(2)} \, . \]
	
Putting all together we obtain 
	\[ \begin{split}  
	\bbE (J_1 \mathds{1}(X_{J_1} = 1_*)) & = 
	\frac{ \xi_0^+ / \G \big( \g^{(1)} \big) + \eta_0^+ / \G \big( \g^{(2)} \big) }{\n_0^2}
	+ \frac{ \xi_0^+ \Psi^{(1)} / \G \big( \g^{(1)} \big) + \eta_0^+ \Psi^{(2)}/ \G \big( \g^{(2)} \big) }{\n_0} \\
	& = \frac{ \xi_0^+}{\n_0 \G \big( \g^{(1)} \big)} \bigg( \frac{1}{\n_0} + \Psi^{(1)} \bigg)
	+ \frac{ \eta_0^+}{\n_0 \G \big( \g^{(2)} \big)} \bigg( \frac{1}{\n_0} + \Psi^{(2)} \bigg) \, . 
	\end{split} \]

By symmetry, then, we also have
	\[  \begin{split}  
	\bbE (J_1 \mathds{1}(X_{J_1} =- 1_*)) & 
	 = \frac{ \xi_0^-}{\n_0 \G \big( \bar{\g}^{(1)} \big)} \bigg( \frac{1}{\n_0} + \bar{\Psi}^{(1)} \bigg)
	+ \frac{ \eta_0^-}{\n_0 \G \big( \bar{\g}^{(2)} \big)} \bigg( \frac{1}{\n_0} + \bar{\Psi}^{(2)} \bigg) \\
	& = 
	\frac{ \xi_0^+ \D^{(1)} }{\n_0 \G \big( \g^{(1)} \big)} \bigg( \frac{1}{\n_0} + \bar{\Psi}^{(1)} \bigg)
	+ \frac{ \eta_0^+ \D^{(2)} }{\n_0 \G \big( \g^{(2)} \big)} \bigg( \frac{1}{\n_0} + \bar{\Psi}^{(2)} \bigg) \, . 
	\end{split} \]
where $\G \big( \bar{\g}^{(i)} \big)$ and $\bar{\Psi}^{(i)}$ are defined as $\G \big( \g^{(i)} \big)$ and $\Psi^{(i)}$ but for the reversed chain $\bar{\g}^{(i)}$, and we have used the relation 
$\G \big( \bar{\g}^{(i)} \big) = \r_0^{(i)} \G \big( \g^{(i)} \big) / \Delta^{(i)}$ (cf. \eqref{supersonico}).

Going back to \eqref{crema}, then, we find
	\begin{equation}\label{pezzo1}
	\begin{split}
	\bbE (J_1 \mathds{1}(X_{J_1} = 1_*)) - \bbE(J_1 \mathds{1}(X_{J_1} = -1_*)) = &
	 \frac{ \xi_0^+}{\n_0 \G \big( \g^{(1)} \big)} \bigg( \frac{1 - \D^{(1)}}{\n_0} + \Psi^{(1)} - 
	 \D^{(1)} \bar{\Psi}^{(1)}\bigg) + \\
	 &  \frac{ \eta_0^+}{\n_0 \G \big( \g^{(2)} \big)} \bigg( \frac{1 - \D^{(2)}}{\n_0} + \Psi^{(2)} - 
	 \D^{(2)} \bar{\Psi}^{(2)}\bigg) \, . 
	 \end{split} 
	 \end{equation}
In order to simplify the above expression, we again use Lemma \ref{7nani} to get
	\begin{equation*}
	\Psi^{(i)} = \sum_{k=1}^{N_i-1} z_k^{(i)} h_k^{(i)} \, , 
	\qquad 
	\D^{(i)} \bar{\Psi}^{(i)} = \sum_{k=1}^{N_i-1} w_k^{(i)} (1-h_k^{(i)})\, ,
	\end{equation*}
where the $z_k^{(i)}$'s and the $w_k^{(i)}$'s have been defined in \eqref{z100}, \eqref{w100}, and here refer to the chain $\g^{(i)}$, $i=1,2$.

In conclusion, in light of the above expressions  and since $z_k^{(i)}+w_k^{(i)}=r_k^{(i)}$,  \eqref{pezzo1} becomes
	\begin{equation}\label{pezzo1bis}
\boxed{
	\begin{split}
	 \bbE (J_1 \mathds{1}(X_{J_1} = 1_*)) & - \bbE(J_1 \mathds{1}(X_{J_1} = -1_*))  \\
	& =\frac{ \xi_0^+}{\n_0 \G \big( \g^{(1)} \big)} \bigg( \frac{1 - \D^{(1)}}{\n_0} + 
	\sum_{k=1}^{N_1-1} ( h_k^{(1)}r_k^{(1)} -w_k^{(1)}) \bigg) \\
	&
	+\frac{ \eta_0^+}{\n_0 \G \big( \g^{(2)} \big)} \bigg( \frac{1 - \D^{(2)}}{\n_0} +
	\sum_{k=1}^{N_2-1}  ( h_k^{(2)}r_k^{(2)} -w_k^{(2)})  \bigg) \, . 
	\end{split} }
	\end{equation}
	We now focus on the computation of $\bbE(J_1^2)$. 
By conditioning on the position of the process after the first jump, and using the Markov property, we find, in analogy with \eqref{chiave},
	\begin{equation} \label{violino}
	\bbE(J_1^2) = \frac{ 1}{\n_0} \big[ \xi_0^+ \bbE_{a_1}(J_1^2) + \xi_0^- \bbE_{a_{N_1-1}}(J_1^2)  \big] 
	+ \frac{ 1}{\n_0} \big[ \eta_0^+ \bbE_{b_1}(J_1^2) + \eta_0^- \bbE_{b_{N_2-1}}(J_1^2)  \big] 
	+  \frac{2}{\n_0} \bbE(J_1) \, . 
	\end{equation}
Note that all expectations appearing in the above formula can be computed as in the linear chain case since, as already pointed out, once the process $X$ picks a chain with the first jump, it does not leave it at least up to time $J_1$.

Hence, in analogy with \eqref{conf1} and \eqref{conf2}, we find
	\begin{equation} \label{pezzo2}
	\begin{split}
	 \bbE_{a_1} (J_1^2) = \frac{\Phi^{(1)}}{	\G \big( \g^{(1)} \big) }, \qquad \qquad \qquad& \qquad 
	 \bbE_{b_1} (J_1^2) = \frac{\Phi^{(2)}}{	\G \big( \g^{(2)} \big) },  \\
	 \bbE_{a_{N_1-1}} (J_1^2) = \frac{\bar{\Phi}^{(1)}}{	\G \big( \bar{\g}^{(1)} \big) }
	= \frac{\D^{(1)} \bar{\Phi}^{(1)}}{	\r_0^{(1)} \G \big( {\g}^{(1)} \big) }\, , \quad & \quad 
	 \bbE_{b_{N_2-1}} (J_1^2) = \frac{\bar{\Phi}^{(2)}}{	\G \big( \bar{\g}^{(2)} \big) }
	= \frac{\D^{(2)} \bar{\Phi}^{(2)}}{	\r_0^{(2)} \G \big( {\g}^{(2)} \big) }\, ,
	\end{split}
	\end{equation}
where $\Phi^{(i)}$ is defined as in \eqref{conf1} but for the chain $\g^{(i)}$, $i=1,2$. 

Another application of Lemma \ref{7nani} yelds
	\begin{equation}\label{viola}
	 \Phi^{(i)} + \D^{(i)} \bar{\Phi}^{(i)} = 
	 \sum_{n=1}^{N_i -1} 2 k_n^{(i)} ( z_n^{(i)} + w_n^{(i)} ) = 
	 \sum_{n=1}^{N_i -1} 2 k_n^{(i)}  r_n^{(i)} \, , 
	 \end{equation}
for $i=1,2$, where $k_n^{(1)} := \bbE_{a_n}(J_1)$ and $k_n^{(2)} := \bbE_{b_n}(J_1)$ can be computed from $k_0^{(i)}:=0$, $k_1^{(i)}$ by mean of the relations
	\[ \begin{split} 
	& k_{n+1}^{(1)} = k_n^{(1)} + \r_n^{(1)} ( k_n^{(1)} - k_{n-1}^{(1)} ) - \frac{1}{\xi_n^+} \\
	& k_{n+1}^{(2)} = k_n^{(2)} + \r_n^{(2)} ( k_n^{(2)} - k_{n-1}^{(2)} ) - \frac{1}{\eta_n^+} \, .
	\end{split} \]
 In particular, $k_k^{(i)}$ can be taken as in Proposition \ref{calzino}, now referred to the chain $\g^{(i)}$.
Plugging \eqref{pezzo2} and \eqref{viola} into \eqref{violino} we obtain
\[
	\bbE(J_1^2)  = \frac{2 }{\n_0} \bbE(J_1) + 
	\frac{ \xi_0^+}{\n_0 \G \big( \g^{(1)} \big)} \Big( \Phi^{(1)} + \D^{(1)} \bar{\Phi}^{(1)} \Big) 
	+ \frac{ \eta_0^+}{\n_0 \G \big( \g^{(2)} \big)} \Big( \Phi^{(2)} + \D^{(2)} \bar{\Phi}^{(2)} \Big) \,.\]
Hence, we get	\begin{equation} \label{pezzo2bis}
\boxed{	\bbE(J_1^2) = \frac{2 }{\n_0} \bbE(J_1) + 
	\frac{ \xi_0^+}{\n_0 \G \big( \g^{(1)} \big)} \Big( \sum_{n=1}^{N_1 -1} 2 k_n^{(1)}  r_n^{(1)} \Big) 
	+ \frac{ \eta_0^+}{\n_0 \G \big( \g^{(2)} \big)} \Big(  \sum_{n=1}^{N_2 -1} 2 k_n^{(2)}  r_n^{(2)}\Big)\,.} 
	\end{equation}

\begin{Conclusion}\label{grifondoro}
The diffusion coefficient $\s^2$ is obtained by the general formula \eqref{crema} where  the addenda in the r.h.s. are specified by \eqref{v_parallel_c}, \eqref{mattone1}, \eqref{mattone2}, \eqref{pezzo1bis}, \eqref{pezzo2bis}, where 
$\D^{(i)}$, $\G(\g^{(i)})$ and $r_k^{(i)} $  are  defined in \eqref{pasqua1}, \eqref{pasqua2} and \eqref{pasqua3}, respectively; $r_+$ and $r_-$ are defined in \eqref{c2}; $\nu_0:= \xi_0^+ + \xi_0^-+\eta_0^++\eta_0^-$;
  $w_k^{(i)}$, $k_k^{(i)}$ and $h_k^{(i)}$ are defined as  in Proposition \ref{calzino}  now referred to the chain $\g^{(i)}$; 

\end{Conclusion}


\section{Examples and comparison with previous results}\label{confronto}

In this section we consider some cross--checks of our formulas (Subsections \ref{mirta1} and \ref{mirta2})  and compare our results with the ones in \cite{K1} for the parallel--chain model, showing that the diffusion coefficient derived in \cite{K1} agrees neither with our computations nor with some general principles. By correcting formula (26) and (28)  of \cite{K1} as explained in the introduction, at least in the special cases consider here we then recover the result of \cite{K1}.

\subsection{Homogeneous random walk on $\bbZ$ with $N$ periodic rates}\label{mirta1} As a  cross--check of Proposition \ref{calzino} we consider 
 the $N$-periodic linear model as in Section \ref{esempi1}, and assume that 
	\[ \xi_i^+ =: \a \, , \quad \xi_i^- =: \b \qquad \forall i=0,1\ldots N \, . \]
As a consequence, $\r_0 = \r_1 = \cdots = \r_N = \frac{\b}{\a} =:\r$. 
 We compute the asymptotic velocity and diffusivity from Proposition \ref{calzino}. 
When  $\a= \b$  it holds $r_k= N/\a$ and $\D=1$, thus implying that $v= 0$ and $\s^2= 2\a /N^2$. Let us now take $\a \not =\b$.
 We have
	\[ \D = \r^N , \quad
	r_k =\frac{1-\r^N}{\a (1-\r )} , \quad
	\Lambda_k = \frac{1-\r^k}{1-\r} , \quad
	\Upsilon_k = \frac{1}{\a (1-\r)} \bigg( k - \frac{1-\r^k}{1-\r} \bigg) \]
for $1 \leq k \leq N$. 
Moreover, we get
	\[ w_k = \frac{\r^{N-k} - \r^N}{ \a (1-\r )} , \quad
	h_k = \frac{1-\r^k}{1-\r^N}, \quad
	k_k = \frac{1}{\a (1-\r)} \bigg( N \frac{1-\r^k}{1-\r^N} -k\bigg) \]
for $1 \leq k \leq N-1$.
Following \eqref{razzo} and \eqref{leopardi}, then, we find
	\[ \begin{split} 
	& \sum_{k=1}^{N-1} 2k_k r_k = \frac{2 (1-\r^N )}{[ \a (1-\r )]^2 } \Big( \frac{N^2}{1 - \r^N} - \frac{N}{1-\r} - \frac{N(N-1)}{2} \Big) \, , 
	\\ &
	\sum_{k=1}^{N-1} (  w_k - h_k r_k ) = \frac{1}{\a (1-\r )} \Big[  \frac{2(1-\r^N)}{1-\r} -N -1 -(N-1)\r^N \Big] \, ,  
	\end{split} \]
from which \[ v = \frac{\a - \b}{N}\,,\;\;\;\;\; \s^2 = \frac{\a + \b}{N^2} \,,\] in agreement with our discussion preceding  Subsection \ref{mezzaluna}.

\subsection{Random walk on $\bbZ$ with $2$ periodic rates}\label{mirta2}
Consider the $N$ periodic linear model as in Section \ref{esempi1}, and set $N=2$. Let $\r_0 = \frac{\xi_0^-}{\xi_0^+}$ and $\r_1 = \frac{\xi_1^-}{\xi_1^+}$. We again compute asymptotic velocity and diffusivity from Proposition \ref{calzino}. 
We have:
	\[ \begin{split}
	& \D = \r_0 \r_1 , \quad
	r_1 = \frac{1}{\xi_1^+}(1+\r_0) , \quad
	r_2 = \frac{1}{\xi_0^+}(1+\r_1) , \quad
	\Lambda_1 = 1, \quad
	\Lambda_2 = 1 + \r_1 , \quad \\ &
	\Upsilon_1 =0, \qquad
	\Upsilon_2 = \frac{1}{\xi_1^+} , \qquad 
	 w_1 = \frac{\r_0}{\xi_1^+} , \qquad
	h_1 = \frac{1}{1+\r_1}, \qquad 
	k_1 = \frac{1}{\xi_1^+(1+\r_1)} \, . 
	\end{split} \] 
According to \eqref{razzo} and \eqref{leopardi}, then, we find 
	\[ \sum_{k=1}^{N-1} 2 k_k r_k = 2 k_1 r_1 = \frac{2}{(\xi_1^+)^2} \frac{1+\r_0 }{1+\r_1}, \qquad 
	\sum_{k=1}^{N-1} (w_k - h_k r_k) =   w_1-h_1 r_1 = \frac{1}{\xi_1^+} \Big(   \r_0 - \frac{1+\r_0}{1+\r_1}\Big) , \]
	from which we obtain
	\begin{align}
	 v& = \frac{ 1 - \r_0 \r_1}{ \frac{1+\r_0}{\xi_1^+} + \frac{1+\r_1}{\xi_0^+} } 	
	= \frac{ \xi_0^+ \xi_1^+ - \xi_0^- \xi_1^-}{\xi_0^+ + \xi_1^+ + \xi_0^- + \xi_1^-} \,,\\
	\s^2&  =  \frac{ 1 + \r_0 \r_1}{\frac{1+\r_0}{\xi_1^+} + \frac{1+\r_1}{\xi_0^+} } 	
	- \frac{ 2(1 - \r_0 \r_1)^2}{\big( \frac{1+\r_0}{\xi_1^+} + \frac{1+\r_1}{\xi_0^+} \big)^3 \xi_0^+ \xi_1^+}  \nonumber\\
	&= \frac{ \xi_0^+ \xi_1^+ +\xi_0^- \xi_1^-}{\xi_0^+ + \xi_1^+ + \xi_0^- + \xi_1^-}
	-2\frac{ (\xi_0^+ \xi_1^+ - \xi_0^- \xi_1^-)^2}{(\xi_0^+ + \xi_1^+ + \xi_0^- + \xi_1^-)^3}
	\, . \label{diff_2}
	\end{align}
The above expressions coincide with the correspondent ones in \cite{D}, equations  (49) and (47) there.

\subsection{Parallel chains with $N=M=2$}
Consider the parallel chains model analysed in Section \ref{esempi3}, and set $N=M=2$. Moreover, let
	\[ \xi_0^+ = \xi_{a_1}^+ = \eta_0^+ = \eta_{b_1}^+ =: \a , \qquad 
	\xi_0^- = \xi_{a_1}^- = \eta_0^- = \eta_{b_1}^- =: \b  \] 
and consequently $\r_0^{(1)} = \r_1^{(1)} = \r_0^{(2)} = \r_1^{(2)}=\b/\a =: \r$. Note that, with this choice, the upper and lower chains are identical. 
We compute the asymptotic velocity according to Proposition \ref{ovetto}.  When   $\a=\b$, it holds   $\rho=\D^{(1)}= \D^{(2)}=1$ and  we get $v=0$. 
When $\alpha \not = \beta$, note that\[\D^{(1)} = \D^{(2)} = \r^2\,,\qquad r_k^{(1)} = r_k^{(2)} = \frac{1}{\a} (1+\r )  \text{ for } k=1,2
\] (see \eqref{pasqua1} and \eqref{pasqua3}), which gives
	\begin{equation}\label{usignolo} v = \frac{2}{3} \a (1-\r ) = \frac{2}{3} (\a - \b ) \, . \end{equation}
We next compute the diffusion coefficient $\s^2$ according to Conclusion \ref{grifondoro}, i.e.  putting together  \eqref{crema} and \eqref{mattone1}, \eqref{mattone2}, \eqref{pezzo1bis}, \eqref{pezzo2bis}. Recall the definition of $r_\pm$ in \eqref{c2}.  First we consider the case $\a=\b$. Since $v=0$
 it holds  $\s^2=(\tilde p + \tilde q )/\bbE (J_1)$ by \eqref{crema} and  Conclusion \ref{grifondoro}. We have $r_+=r_-=\a$, $r^{(1)}_k= r^{(2)}_k=2/\a$ and by \eqref{mattone1} we conclude that $\s^2=2\a/3$. We now consider the case $\a \not = \b$ (i.e. $\r\not =1$).
 We have:
	\[  \begin{split} 
	& r_+ = \frac{2\a}{1+\r} , \qquad
	r_- =   \frac{2\a \r^2}{1+\r} , \qquad
	\G (\g^{(1)}) = \G (\g^{(2)}) = 1+\r , \qquad  \\
	& h_1^{(1)} = h_1^{(2)} = \frac{1}{1+\r }, \quad \qquad 
	w_1^{(1)} = w_1^{(2)} = \frac{\r}{\a} , \quad \qquad 
	k_1^{(1)} = k_1^{(2)} = \frac{1}{\a (1+\r )} \, .
	\end{split} \] 
Hence we can write down the different terms in \eqref{crema}, namely:
	\[ \begin{split} 
	& \bbE (J_1) =  \frac{3}{2\a (1+\r )} , \qquad \frac{ \tilde{p} + \tilde{q}}{\bbE (J_1) } = \frac{ 2\a (1+\r^2)}{3(1+\r )} 
	, \qquad
	\bbE (J_1^2) = \frac{\bbE (J_1)}{\a (1+\r )} + \frac{2}{[ \a (1+\r )]^2} ,\\
	&  \bbE (J_1 \mathds{1} (X_{J_1} = 1_*)) - \bbE (J_1 \mathds{1} (X_{J_1} = -1_*)) = 
	\frac{3(1-\r)}{2\a (1+\r )^2} , 
	\end{split} \]
from which we obtain
	\begin{equation}\label{falchetto}\s^2 = \frac{2\a}{27 (1+\r )} (5\r^2 + 8\r + 5)=\frac{10 \a^2 + 16 \a \b + 10 \b^2}{27(\a+\b)} \, . \end{equation}
	 Note that the above expression covers also the case $\a=\b$.

We point out that the case of identical upper and lower chains can be reduced to the linear case via a change of rates. We use a lumping procedure considering the map 
$f: \cV \mapsto \bbZ$ mapping points $n_*$ to $2n$ for $n \in \bbZ$, and mapping $a_1^{n
}$ and $b_1^{n}$ to $2n+1$.  
 In general, the image of a Markov chain is not a Markov chain.
 In our case , given $k \in \bbZ$ and given $ x \in f^{-1}(k)$, the sum  $\sum _{y \in f^{-1}(k\pm 1) }r(x,y)$ does not depend on $x$. It then follows  
(cf. e.g. \cite{TK}[Theorem 2.11]) that 
   $( f(X_t) )_{t \in \bbR_+}$ is indeed a continuous--time  Markov chain on $\bbZ$ with periodic rates of periodicity $2$, given by 
   $ \bar  r(k, k\pm 1):=\sum _{y \in f^{-1}(k\pm 1) }r(x,y)$ where $x$ is an arbitrary state in $f^{-1}(k)$. Using the notation for the random walks on $\bbZ$ with $2$ periodic rates, we have 
     $\xi_0^+ = 2\a $, $\xi_0^- = 2\b $, $\xi_1^+ = \a$, $\xi_1^- = \b$. Note that $\r_0 = \r_1 = \b / \a$.
      By applying \eqref{diff_2}, then, we find 
	\[\s^2 =  
	 \frac{2\a^2+2\b^2}{3\a+3\b} - 2\frac{(2 \a^2 - 2 \b^2)^2}{(3\a+3 \b)^3}=
	 \frac{10 \a^2 + 16 \a \b + 10 \b^2}{27(\a+\b)} \,,  \]
	 in agreement with our result  \eqref{falchetto} for the parallel-chain model.

We now compare our formula with the one obtained in \cite{K1}[Section II]. We take $d=2$ there, $d$ being the distance between  neighboring sites $n_*$  in \cite{K1}. The diffusion constant $D$ given in \cite{K1} is defined as $\frac{1}{2} \lim_{t \to \infty}
\frac{d}{dt} [ \langle x^2 (t) \rangle - \langle x(t) \rangle ^2 ]$,  where $x(t)$ corresponds to our $f(X_t)$. Our diffusion constant $\s^2$ formally corresponds to $\lim _{t \to \infty} \frac{1}{t} [ \langle (X_t^*)^2  \rangle - \langle X^*_t \rangle ^2] $. Since $|X^*_t- f(X_t)/2| \leq 1$, it must be $D= \frac{\s^2}{2}$. We now show that 
\begin{equation}\label{primula}
D\not = \frac{\s^2}{2}=  \frac{5 \a^2 +8 \a \b + 5\b^2}{27(\a+\b)} \,.
\end{equation}
Note that the last member is the correct value due to the lumping procedure.

\medskip

To check \eqref{primula} we compute $D$ using  the notation of  \cite{K1}[Section II]. We have  $D = D_0 + D_1 + D_2 + D_3$, and we now compute the constants $D_i$'s.  We treat the case $\a \not = \b$.
  Note that $N=M=2$, $ u_i=\a_i=\a$, $w_i= \b_i=\b$. In particular, equations (3) and (4) in \cite{K1} read
$$ \P_{(0)j}^k=\P_{(1)j}^k=\P_{(0)j}^{\dag,k}=\P_{(1)j}^{\dag,k}=\left( \frac{\b}{\a}\right)^{k -j+1} = \rho^{k -j+1} \,. $$
Formulas (9), (10) in \cite{K1} read
\[ R_N=R_M=\frac{2}{\a} (1+\rho) \,, \qquad r_j^{(0)}= r_j^{(1)}=\frac{1}{\a} (1+\rho)  \,.
\]
Formulas (7), (8) in \cite{K1} read
\[ V_0=V_1= \frac{2 \a}{1+\rho}(1-\rho^2)\frac{1}{3}=\frac{2\a(1-\rho)}{3}=\frac{2(\a-\b)}{3}\,,\]
thus implying that $V=V_0+V_1= (4/3) ( \a-\b)$ due to (6) in \cite{K1}. Recall that $V= \lim _{t \to \infty}\frac{d}{dt} \langle x(t) \rangle $ in \cite{K1}. 
Comparing with \eqref{usignolo} we get that $V= 2 v$  as it must be.

We now note that formulas (16),(17), (18), (19) in \cite{K1} read 
\[
s_j^{(0)}= \frac{1}{\a} (1+ \rho)=:s\,, \;\;\;\; b_j^{(0)}= b_j^{(1)}=\frac{1}{3}\,, \;\;\;\; U_N= \frac{2}{3}(1+\rho)\,, \;\;\;\; S_N= \frac{2}{\a}(1+\rho) \,.
\]
Formula (15) in \cite{K1} then becomes
\[
 J_0=\frac{1}{2} s 2 (\a-\b) \frac{1}{3}+\frac{2}{3}(1+\rho)+\frac{2}{3}\a (1-\rho) \frac{2}{\a}(1+\rho) -
 \frac{2}{3}\a (1-\rho)s\frac{1}{3}=\frac{13}{9} (1-\rho^2)+\frac{2}{3} (1+\rho)\,.
 \]
 We conclude that (14) in \cite{K1} reads
 \begin{equation}\label{dino0} D_0= V \frac{2}{3}- V \cdot 2 + V \frac{J_0}{1-\rho^2}=  (4/3)\a ( 1-\rho)\left[\frac{2}{3}-2  
 +\frac{13}{9}+ \frac{2}{3} \frac{1}{1-\rho}\right]=\frac{28}{27}\a-\frac{4}{27}\b \,.
 \end{equation}
 Since the lower and the upper chains are identical, the constant $D_1$ defined in (21) of \cite{K1}
equals $D_0$, while the constant $D_3$ defined in (27) of  \cite{K1} is zero. We then move to $D_2$ defined in (25) of \cite{K1}. To compute it we observe that  (26) in \cite{K1} reads
\[
J_2= \left[ \frac{1}{3} + \frac{V}{2\a}\right]+ \left[ \frac{\rho}{2\a} 2 (\a-\b)\frac{1}{3} + \frac{\rho}{3} + \frac{V \rho}{2 \a} \frac{1}{3}\right]= \frac{9+2 \rho - 5 \rho^2}{ 9}\,,
\]
where, in the second member,  the two addenda inside the  $[\cdot ]$--brackets correspond to $j=0$ and $j=1 $ in \cite{K1}, respectively.  We then get
\footnote{We point out that in the definition of $J_2$  (cf. (25) in \cite{K1}) the term $\P ^0_{(1)1}$ appears although not defined in  (4) of \cite{K1}. Indeed,  $\P ^0_{(1)1}=1$ (in agreement with the rule that a product on an empty set is one). This can be checked as follows. The terms $D_2$ comes from the last addendum in \cite{K1}[eq. (A44)] involving $T^{(1)}_0$ defined in \cite{K1}[eq. (A39)]. By comparing this term with $D_2$, and therefore $J_2$,   one gets that in \cite{K1}  $T_0^{(1)}$ can be thought as $-\frac{1}{\a_0} \sum _{k=0}^{M-1} y_k^{(1)} \P _{(1)1}^k/ [1-\P ^M_{(1)1}]$, $y_k^{(1)}$ being defined in \cite{K1}[eq. (A37)], with the convention that  $\P ^0_{(1)1}=1$.}, by (25) in \cite{K1},
\begin{equation}\label{dino2}D_2=- \frac{V J_2}{1-\rho^2}=   -\frac{4}{3} \a \frac{J_2}{1+\rho}=
\frac{4\a}{27} \frac{ 5 \rho^2-2 \rho -9}{1+\rho}\,.
\end{equation}

Finally, by (13) in \cite{K1}, we have (recall \eqref{dino0} and \eqref{dino2})
\begin{equation}\label{dino3}
\begin{split} D&= D_0+D_1+D_2+D_3= 2 D_0+ D_2= \frac{4\a}{27}\left[ (14-2\rho)+\frac{ 5 \rho^2-2 \rho -9}{1+\rho}\right]\\&= \frac{4\a}{27} \frac{3 \rho^2+4\rho+5}{1+\rho}
 =\frac{20 \a^2+16 \a \b + 12 \b^2}{27(\a+\b)}\,,
\end{split}
\end{equation}
thus proving \eqref{primula}.

We point out that when $\b \to 0$, we get $D= (20/27) \a$ in agreement with (31) in \cite{K1}. Due to \eqref{falchetto} in this special case the relation $D= \s^2/2$ is verified.

On the other hand,  replace   in (26) of \cite{K1}  the term    $ \sum _{i=0}^{N-1} b_i^{(0)}$ with 
$ \sum _{i=1}^{N-1} b_i^{(0)}$. Then one would get
\[
J_2= \left[ \frac{1}{3} + \frac{V}{2\a}\right]+ \left[ \frac{\rho}{2\a} 2 (\a-\b)\frac{1}{3} + \frac{\rho}{3} + \frac{V \rho}{2 \a} \frac{2}{3}\right]= \frac{9+4 \rho - 7 \rho^2}{ 9}\,.
\]By correcting \eqref{dino2} and \eqref{dino3} as consequence, one gets  an identity in \eqref{primula}. We have indeed checked that our general formulas confirm the corrected version of \cite{K1} in the case $M=N=2$ and generic rates.

\section{Proof of Propositions \ref{virulino1} and \ref{virulino2}} \label{EA}
In this section we collect the proofs of Propositions \ref{virulino1} and \ref{virulino2}.

\subsection{Proof of Proposition \ref{virulino1}} Take  a pair $x,y$ in $\tilde V$ and   and let  $\g=(x_0,x_1, \dots, x_n)$   be a linear chain in $\cC\cup \bar \cC$  from $x=x_0$ to $y=x_n$, if it exists. Recalling \eqref{aereo}, 
by the linear system \eqref{febbre1} it holds
\[  r_i^- \bigl
[ 
\phi(x_i) -\phi(x_{i-1}) \bigr]+  r_i^+\bigl[\phi(x_i)- \phi (x_{i+1})\bigr]=0 \,,  \qquad  1\leq i \leq n-1\,.
\]
The above system of equations trivially leads to the identity $\phi(x_{i+1})-\phi(x_i)= c_i \bigl(\phi(x_1)- \phi(x_0)\bigr)$ for all  $i : 1\leq i \leq n-1$, where 
\begin{equation}\label{cipcip}
 c_i:= \begin{cases}
1 & \text{ if } i=0\,,\\
\frac{r_1^-}{r_1^+}\frac{r_2^-}{r_2^+} \cdots \frac{r_i^-}{r_i^+} & \text{ if } 1\leq i \leq n-1\,.
\end{cases}
\end{equation}
Note that $\sum_{i=0}^{n-1} c_i = \G(\g)$ defined in \eqref{alexian}.
 By a telescoping argument, one gets that 
\begin{equation}\label{film_epic}\phi(x_i) =\phi(x_0)+\bigl( \sum _{j=0}^{i-1} c_j\bigr)\bigl ( \phi(x_1)- \phi(x_0)\bigr)\,, \qquad 1\leq i \leq n \,.
\end{equation}
Taking $i=n$ in \eqref{film_epic}  one  obtains $\phi(x_1)- \phi(x_0)$ in terms of $\phi(x_n)- \phi(x_0)=\phi(y)-\phi(x)$ and therefore
\begin{equation}\label{pierre} \phi (x_i)= \phi (x) + \frac{\sum_{j=0}^{i-1}c_i}{\G(\g)} \bigl[ \phi (y)- \phi(x) \bigr]\,.
\end{equation}
In general, if there are more linear chains $\g^{(1)}, \g^{(2)}, \dots, \g^{(k)}$ from $x $ to $y$  in $\cC \cup \bar \cC$, then (using the  notation 
introduced before Prop. \ref{virulino1}) it must be
\begin{equation}\label{pierreZZZ} \phi (x^{(s)} _i )= \phi (x) + \frac{\sum_{j=0}^{i-1}c_i^{(s)} }{\G\bigl(\g^{(s)} \bigr)} \bigl[ \phi (y)- \phi(x) \bigr]\,,
\end{equation}
where $c_i^{(s)}$ is the above defined constant $c_i$ referred to $\g= \g^{(s)}$.
By the linear system \eqref{febbre1} it holds
\[ \bigl[   \sum _{z: (x,z) \in \tilde E } \tilde r (x,z)\bigr]\phi(x)=  \sum _{\substack{ z \in \tilde V\setminus \{x_1^{(1)}, \dots,x_1^{(k)}  \} \,:\\ (x,z) \in \tilde E}}\tilde r (x, z) \phi(z) +\sum _{s=1}^k  \tilde r(x,x_1^{(s) }) \phi(x_1^{(s)})\,.
\]
Due to  \eqref{pierreZZZ}    the above identity is equivalent to 
\begin{equation}\label{libro1}
\begin{split}
\a  \phi(x)& = \sum _{\substack{ z \in \tilde V\setminus \{x_1^{(1)}, \dots, x_1^{(k)} \} \,:\\ (x,z) \in \tilde E}}\tilde  r (x, z) \phi(z) +\left[\sum _{s=1}^k  \frac{\tilde r(x,x_1^{(s) })}{ \G\bigl( \g^{(s)} \bigr)} \right] \phi(y) \\
&=  \sum _{\substack{ z \in \tilde V\setminus \{x_1^{(1)}, \dots, x_1^{(k)} , y\} \,:\\ (x,z) \in \tilde E}}\tilde  r (x, z) \phi(z) +\bar r(x,y) \phi(y)\,,
\end{split}
\end{equation}
where $\a $ equals the sum of the coefficients in the linear combinations of $\phi$ in the r.h.s. Repeating the above procedure   as $y$ varies in $\tilde V$ we get the first equation in \eqref{febbre1bis}. The remaining equations $\phi(-1_*)=0$ and $\phi(1_*)=1$  are trivially satisfied.

\subsection{Proof of Proposition  \ref{virulino2}} Take  a pair $x,y$ in $\tilde V$ and   and let  $\g=(x_0,x_1, \dots, x_n)$  be a linear chain in $\cC\cup \bar \cC$ from $x$ to $y$, if it exists. 
  Recalling \eqref{aereo}, 
by the linear system \eqref{febbre2} it holds
\[  r_i^- \bigl[ \psi(x_i)-
\psi(x_{i-1})\bigr] + r_i^+\bigl[\psi(x_i)-\psi (x_{i+1})\bigr]  =1 \,, 
\]
for any $ i : 1\leq i \leq n-1$.
Recall \eqref{cipcip}. The above system of equations trivially leads to the identity 
\[ \psi(x_{i+1})-\psi(x_i)=c_i \bigl[ \psi(x_{1})-\psi(x_{0})\bigr]-
 \sum _{j=1} ^i \frac{1}{r_j^+} \frac{r_{j+1}^-}{r_{j+1}^+} \cdots \frac{r_i^-}{r_i^+}
\]
for $1\leq i \leq n-1$, where by convention    the last term  equals $1/r_1^+$ when $j=1$.
 By a telescoping argument, one gets that 
\begin{equation*}
\psi(x_i) =\bigl( \sum _{m=0}^{i-1} c_m\bigr)\bigl [\psi(x_1)- \psi(x_0)\bigr]+\psi(x_0)- \sum_{m=1}^{i-1} \sum _{j=1} ^m \frac{1}{r_j^+} \frac{r_{j+1}^-}{r_{j+1}^+} \cdots \frac{r_m^-}{r_m^+}  \,, \qquad 1\leq i \leq n \,,
\end{equation*}
where by convention  the sum $\sum_{m=1}^{i-1}$ is set equal to zero if   $i=1$.
Taking $i=n$ in the above identity and recalling \eqref{alexian}, one  obtains 
\begin{equation}\label{pierre_bis} 
\psi (x_1)- \psi (x) = \frac{1}{\G(\g)} \bigl[ \psi (y)- \psi(x) \bigr]  + \frac{1}{\G(\g)}  \sum_{1\leq j \leq m \leq n-1} \frac{1}{r_j^+} \frac{r_{j+1}^-}{r_{j+1}^+} \cdots \frac{r_m^-}{r_m^+}   \,.
\end{equation}
Recalling the definition of $c(\g)$ given in \eqref{rachele}
we get 
\begin{equation}\label{pierre_bis1} \psi (x_1)= \psi (x) + \frac{1}{\G(\g)} \bigl[ \psi (y)- \psi(x) \bigr]  + \frac{c(\g) }{\G(\g)}  \,.
\end{equation}

Let  $\g^{(1)}, \g^{(2)}, \dots, \g^{(k)}$ be the linear chains in  $\cC\cup \bar \cC$ from  $x$ to $y$. Recall the notation introduced before Proposition \ref{virulino1}. By the linear system \eqref{febbre2} it holds
\[ \bigl[   \sum _{z: (x,z) \in \tilde E } \tilde r (x,z)\bigr]  \psi(x)=  1+\sum _{\substack{ z \in \tilde V\setminus \{x_1^{(1)}, \dots,x_1^{(k)}  \} \,:\\ (x,z) \in \tilde E}}\tilde r (x, z) \psi(z) +\sum _{s=1}^k  \tilde r(x,x_1^{(s) }) \psi(x_1^{(s)})\,.
\]
Due to  \eqref{pierre_bis1}    the above identity is equivalent to 
\begin{equation}\label{libro1_bis}
\begin{split}
\a  \psi(x)& = 1+\sum _{\substack{ z \in \tilde V\setminus \{x_1^{(1)}, \dots, x_1^{(k)} \} \,:\\ (x,z) \in \tilde E}}\tilde  r (x, z) \psi(z) +\left[\sum _{s=1}^k  \frac{\tilde r(x,x_1^{(s) })}{ \G\bigl( \g^{(s)} \bigr)} \right] \psi(y)+ \sum _{s=1}^k \frac{ c(\g^{(s)})}{\G( \g^{(s)} )}  \\
&=  1+\sum _{\substack{ z \in \tilde V\setminus \{x_1^{(1)}, \dots, x_1^{(k)} , y\} \,:\\ (x,z) \in \tilde E}}\tilde  r (x, z) \psi(z) +\bar r(x,y) \psi(y)+ \sum _{s=1}^k \tilde r (x, x_1^{(s)} ) \frac{ c(\g^{(s)})}{\G( \g^{(s)} )}  \,,
\end{split}
\end{equation}
where $\a $ equals the sum of the coefficients in the linear combinations of $\psi$ in the r.h.s.
Repeating the above procedure   as $y$ varies in $\tilde V$ we get the first equation in \eqref{febbre2bis}, the remaining equations $\psi (1_*)=0$ and $\psi(-1_*)=0 $  are trivial.

\section{Proof of Lemmata  \ref{calcolo_v} and \ref{calcolo_diff_coeff} }\label{montagnola}

\subsection{Proof of Lemma \ref{calcolo_v}}
Trivially  the expression $v= (\tilde{p}- \tilde{q} )/\bbE(J_1)$ follows from \eqref{trenino} and 
 \eqref{alba_bella1}.
Writing  the event $ \{ X_S = \pm 1_* \}  $ as
	$\bigcup_{k=0}^\infty \{ X_{J_0} = \ldots = X_{J_k} = 0_* , X_{J_{k+1}} = \pm 1_* \}$, 
	using the strong Markov property at times $J_i$ 
we get
	\begin{equation*} 
	\bbP ( X_S = \pm 1_* ) = 
	\sum_{k=0}^\infty \bbP (X_{J_1} = 0_* )^k \bbP ( X_{J_1} = \pm 1_* ) 
	= \frac{\bbP ( X_{J_1} = \pm 1_* ) }{\bbP ( X_{J_1} =  1_* ) + \bbP ( X_{J_1} = -1_* ) } \, . 
	\end{equation*}
This proves the first two identities in \eqref{alba_bella1}.
Similarly, we can write
	\begin{equation} \label{Sdecomp} 
	S = \sum_{k=0}^\infty \mathds{1} ( X_{J_0} = \ldots = X_{J_k} = 0_* , X_{J_{k+1}} \in \{ -1_* , 1_* \}  ) J_{k+1} \, .
	\end{equation}
Taking expectations  and  using the strong Markov property at times $J_i$ 
we get
	\begin{equation} \label{pepe}
	\begin{split}
	\bbE (S)  & = \bbE ( S \mathds{1} ( X_S = - 1_* ) ) + \bbE ( S \mathds{1} ( X_S = 1_* ) ) 
	\\ & 
	= \sum_{k=0}^\infty \bbP ( X_{J_1} =  0_* )^k  \bbP ( X_{J_{1}} = -1_*   ) 
	\big[ k \bbE ( J_{1} | X_{J_1} = 0_* ) +  \bbE ( J_{1} | X_{J_1} = -1_*  ) \big] \\ & \quad 
	+ 
	\sum_{k=0}^\infty \bbP ( X_{J_1} =  0_* )^k  \bbP ( X_{J_{1}} = 1_*   ) 
	\big[ k \bbE ( J_{1} | X_{J_1} = 0_* ) +  \bbE ( J_{1} | X_{J_1} = 1_*  ) \big] 
	\\ & 
	 =  \frac{ \bbE (J_1 \mathds{1} ( X_{J_1} = 0_* )) (\tilde{p} + \tilde{q})}{(\tilde{p} + \tilde{q})^2}+\frac{ \bbE (J_1 \mathds{1} ( X_{J_1} = 1_* ))}{\tilde{p} + \tilde{q}}  + \frac{ \bbE (J_1 \mathds{1} ( X_{J_1} = -1_* ) )}{\tilde{p} + \tilde{q}} 
	  \\ & 
	 = \frac{ \bbE (J_1 )}{\tilde{p} + \tilde{q}} \, ,
	\end{split}
\end{equation}
thus concluding the proof of \eqref{alba_bella1} (above we used that $\sum_{k=0} ^\infty k\g^k= \g/(1-\g)^2$ for $\g \in [0,1)$).
\subsection{Proof of Lemma \ref{calcolo_diff_coeff} } By Theorem \ref{zecchino} it holds
 	\begin{equation}\label{gelato} 
 	\s^2 = \frac{{\rm Var}(X_S^* -vS)}{\bbE (S)} 
 	= \frac{ {\rm Var}(X_S^*) + v^2 {\rm Var}(S) - 2v {\rm Cov}(X_S^*, S)}{\bbE(S)} \, . 
 	\end{equation}
Recall \eqref{alba_bella1}.
 Since $X_S^* = \mathds{1}(X_S = 1_*) - \mathds{1}(X_S = -1_*)$,  we get 
 $ {\rm Var}(X_S^*)  = 
1 - (\tilde{p} - \tilde{q})^2 /(\tilde{p} + \tilde{q})^2$.
 For the covariance observe that  $\bbE (X_S^*) \bbE (S) = \frac{(\tilde{p} - \tilde{q})\bbE(J_1)}{(\tilde{p} + \tilde{q})^2}$,  and by \eqref{pepe} it holds
	\begin{equation*}
	 \begin{split}
	\bbE (X_S^*S) & = \bbE ( S \mathds{1} ( X_S =  1_* ) ) - \bbE ( S \mathds{1} ( X_S = -1_* ) ) \\
	&
	= \frac{\bbE ( J_1 \mathds{1} ( X_{J_1} =  1_* ) ) - \bbE ( J_1 \mathds{1} ( X_{J_1} = -1_* ) )  }{\tilde{p} + \tilde{q}} 
	+ 	\frac{\bbE ( J_1 \mathds{1} ( X_{J_1} =  0_* ) ) (\tilde{p} - \tilde{q}) }{(\tilde{p} + \tilde{q})^2} 
	\, . 
	\end{split}
	\end{equation*}
Hence 
\begin{equation}\label{faticoso2}
\begin{split}
 {\rm Cov}(X_S^* , S)& =\frac{\bbE ( J_1 \mathds{1} ( X_{J_1} =  1_* ) ) - \bbE ( J_1 \mathds{1} ( X_{J_1} = -1_* ) )  }{\tilde{p} + \tilde{q}} 
	\\ &+ 	\frac{\bbE ( J_1 \mathds{1} ( X_{J_1} =  0_* ) ) (\tilde{p} - \tilde{q}) }{(\tilde{p} + \tilde{q})^2}   -\frac{(\tilde{p} - \tilde{q})\bbE(J_1)}{(\tilde{p} + \tilde{q})^2}
\,.
\end{split}
\end{equation}
\medskip

We now concentrate on   ${\rm Var}(S) = \bbE(S^2) - \bbE(S)^2 = \bbE(S^2) - \frac{ \bbE(J_1)^2}{(\tilde{p} + \tilde{q})^2}$. To compute 
  $\bbE(S^2)$ we start with the decomposition $S = S \mathds{1}(X_S = 1_* ) + S  \mathds{1}(X_S = -1_* ) $.
Recall equation \eqref{Sdecomp}. Multiplying everything by $\mathds{1}(X_S = 1_* ) $ and squaring both sides one has: 
	\[ S^2 \mathds{1}(X_S = 1_* ) 
	= \sum_{k=0}^\infty \mathds{1} ( X_{J_0} = \ldots = X_{J_k} = 0_* , X_{J_{k+1}} =1_*  ) J^2_{k+1} \, . \]
Taking the expectation both sides, and using the strong Markov property at times $J_i$  and the fact that
$ \sum_{k=1}^\infty k(k-1) \g^k = 2 \g^2/ (1-\g)^3$ for $\g \in [0,1)$, we get 
	\begin{align*}
	\bbE (S^2  & \mathds{1}(X_S = 1_* ) )   = \sum_{k=0}^\infty \bbP ( X_{J_1} = 0_* )^k  \bbP ( X_{J_1} = 1_* ) 
	\bigg[ k \bbE ( J_1^2 | X_{J_1} = 0_* ) + \bbE ( J_1^2 | X_{J_1} = 1_* ) \\  &
	+ k(k-1) \bbE ( J_1 | X_{J_1} = 0_* ) ^2 + 2k \bbE ( J_1 | X_{J_1} = 0_* ) \bbE ( J_1 | X_{J_1} = 1_* ) \bigg] 
	\\ & 
	 =\frac{\tilde{p} ( 1 - \tilde{p} - \tilde{q} ) }{(\tilde{p} + \tilde{q})^2} \bbE ( J_1^2 | X_{J_1} = 0_* ) 
	+ \frac{\tilde{p} }{\tilde{p} + \tilde{q}} \bbE ( J_1^2 | X_{J_1} = 1_* ) 
	\\
	& + \frac{ 2\tilde{p} ( 1 - \tilde{p} - \tilde{q} )^2 }{(\tilde{p} + \tilde{q})^3} \bbE ( J_1 | X_{J_1} = 0_* ) ^2 
	+  \frac{ 2\tilde{p} ( 1-\tilde{p}-\tilde{q}) \bbE ( J_1 | X_{J_1} = 0_* ) \bbE ( J_1 | X_{J_1} = 1_* ) }{(\tilde{p} + \tilde{q})^2}\\
	& =\frac{\tilde{p}  }{(\tilde{p} + \tilde{q})^2} \bbE ( J_1^2\mathds{1} (X_{J_1}=0_*) ) 
	+ \frac{1 }{\tilde{p} + \tilde{q}} \bbE ( J_1^2 \mathds{1} (X_{J_1}=1_*)  ) \\
	& + \frac{ 2\tilde{p} }{(\tilde{p} + \tilde{q})^3} \bbE ( J_1 \mathds{1}(X_{J_1} = 0_*) ) ^2 +  \frac{ 2 \bbE ( J_1 \mathds{1}( X_{J_1} = 0_*) ) \bbE ( J_1 \mathds{1}(X_{J_1} = 1_* ) )}{(\tilde{p} + \tilde{q})^2}	\, . 
	\end{align*} 
Similarly (exchanging $1_*$ with $-1_*$, and $\tilde{p}$ with $\tilde{q}$) we get
\begin{align*}
	\bbE (S^2  & \mathds{1}(X_S =- 1_* ) )   =
\\
& =\frac{\tilde{q}  }{(\tilde{p} + \tilde{q})^2} \bbE ( J_1^2\mathds{1} (X_{J_1}=0_*) ) 
	+ \frac{1 }{\tilde{p} + \tilde{q}} \bbE ( J_1^2 \mathds{1} (X_{J_1}=- 1_*)  ) \\
	& + \frac{ 2\tilde{q} }{(\tilde{p} + \tilde{q})^3} \bbE ( J_1 \mathds{1}(X_{J_1} = 0_*) ) ^2 +  \frac{ 2 \bbE ( J_1 \mathds{1}( X_{J_1} = 0_*) ) \bbE ( J_1 \mathds{1}(X_{J_1} = -1_* ) )}{(\tilde{p} + \tilde{q})^2}	\, . 
	\end{align*} 
Hence we get
	\begin{equation*}
	\bbE (S^2) = \frac{ \bbE (J_1^2 )}{\tilde{p} + \tilde{q} } 
	+ \frac{ 2  \bbE (J_1 \mathds{1} (  X_{J_1} = 0_* ))  \bbE(J_1) }{(\tilde{p} + \tilde{q})^2  } \, , 
	\end{equation*}
	and
	\begin{equation}\label{faticoso3}
	{\rm Var} (S) = \frac{ \bbE (J_1^2 )}{\tilde{p} + \tilde{q} } 
	+ \frac{ 2  \bbE (J_1 \mathds{1} (  X_{J_1} = 0_* ))  \bbE(J_1) }{(\tilde{p} + \tilde{q})^2  } -
	\frac{ \bbE(J_1)^2}{(\tilde{p} + \tilde{q})^2}\,.
	\end{equation}
	Coming back to \eqref{gelato} and 
using   Lemma \ref{calcolo_v}, the above expression for
$ {\rm Var}(X_S^*) $, 
  \eqref{faticoso2} and \eqref{faticoso3}, after some  simple  calculations  one gets \eqref{crema}.

\appendix

\section{Random time change of cumulative processes}\label{codina}

Consider a sequence $( w_i, \t_i )_{i \geq 1}$ of i.i.d.   2d vectors with values in
$\bbR \times (0,+\infty)$.  For each 
integer $ m \geq 1$ we define 
\begin{align}
&W_m:= w_1+w_2 + \cdots + w_m\,,\\
&\cT_m:= \t_1+\t_2 + \cdots+ \t_m\,.
\end{align}
We set $W_0=\cT_0=0$.
Note that 
$ \lim _{m \to \infty} \cT_m= +\infty$ a.s.  As a consequence, we can univocally define a.s. a random process $\left\{\nu(t)   \right\}_{t \in \bbR_+}$ with values in $\{0,1,2,3, \dots\}$ such that
\begin{equation}\label{warwick}
\cT_{\nu(t)}\leq t < \cT_{\nu(t)+1 }\,, \qquad t \geq 0 
\,.
\end{equation}
Note that $\nu(t)= \max \{ m\in \bbN :\, \cT_m \leq t \}$.
Finally, we define the process $Z:[0,\infty) \to \bbR$ as
\begin{equation}\label{zorro} Z_t:= W_{\nu(t) }\,.
\end{equation}
Note that $Z_0=0$. 
The resulting process $Z= (Z_t)_{t \in \bbR_+}$   is therefore obtained from the cumulative process $(W_m)_{m \geq 0}$ by  a random time change, and  generalizes the concept of (time--homogeneous)  random walk on $\bbR$. For example, if $w_i $ and $\t_i$ are independent and $\t_i$ is an exponential variable of parameter $\l$, then the process $Z$ is a continuous time random walk with exponential holding times of parameter $\l$ and 
with jump probability given by the law of $w_i$. If $\t_i \equiv 1$ for all $i$, then $Z_t= W_{\lfloor t \rfloor }$ ($\lfloor \cdot \rfloor$ denoting the integer part) and  $(Z_n)_{n \in \bbN}$ 
 is a discrete time  random walk on $\bbR$ with jump probability given by the law of $w_i$.

\smallskip

The skeleton process $X^*$ is indeed a special case of  process $Z$ (recall the definition of the random time $S$ given in \eqref{pietro}):
\begin{Lemma}\label{torlonia}
Consider a sequence $( w_i, \t_i )_{i \geq 1}$ of i.i.d.   vectors, with the same law of  the random  vector $\bigl(X^*_S,S\bigr)\in \{-1,1\} \times (0,+\infty)$   when the process   $(X_t)_{t \in \bbR_+}$ 
starts at $0_*$. We define $(Z_t )_{t \in \bbR_+}$ as the stochastic process built from  $( w_i, \t_i )_{i \geq 1}$ according to \eqref{zorro}. Then 
$(Z_t)_{t \in \bbR_+}$ has the same law of $( X_t^*)_{t \in \bbR_+} $ with $X^*_0= 0$. \end{Lemma}
The proof of the above lemma is  very simple and therefore omitted.

\smallskip

We state  our main results for $(Z_t)_{t \in \bbR_+}$:
\begin{TheoremA}\label{teo1}  The stochastic process $(Z_t)_{t \in \bbR_+} $ satisfies:

\begin{itemize}
\item[(i)] [LLN]  Assume that  $ \bbE( \t_i) < \infty$. Then almost 
 surely $ \lim_{t \to \infty} \frac{Z_t}{t} = v:= \frac{ \bbE (w_i) }{\bbE (\t_i)}$.

\item[(ii)] [Invariance Principle]   Assume that $\bbE( w_i^2), \bbE( \t_i^2) < \infty$.  
Given $n\in\{1,2, \dots\}$ define the rescaled process 
$$ B^{(n)} _t := \frac{1}{\sqrt{n}}  \left\{ Z_{ nt} - vnt  \right\} 
$$ in the Skohorod space  $ D( \bbR_+; \bbR)$.  Then as $n\to \infty$ the rescaled process $B^{(n)}$ weakly converges to a Brownian motion on $\bbR$ with diffusion constant
\begin{equation}\label{diffondo}
\s^2:=\frac{ {\rm Var}(w_1-v\tau_1)}{ \bbE(\t_1) }\,.
\end{equation}

\end{itemize}
\end{TheoremA}

\subsection{Proof of the Law of Large Numbers in Theorem \ref{teo1}} The proof is rather standard, we give it for completeness as  short. Since $\lim _{m \to \infty} \cT_m= \infty$ a.s., we have $\lim_{t \to \infty} \nu(t)= \infty$ a.s. Hence from the LLN for $(W_n)_{n \geq 1}$ and $(\cT_n )_{ n \geq 1}$  we deduce that
$\lim _{t\to \infty} \frac{ W_{\nu(t) }}{\nu(t)} =\bbE (w_i)$  and  $ \lim _{t \to \infty} \frac{\cT_{\nu(t)}}{ \nu(t) }  = \bbE (\t_i)$ a.s.  From the last limit and the bounds (recall \eqref{warwick})
\begin{equation*}
\frac{\cT_{\nu(t)}}{ \nu(t) } \leq \frac{t}{\nu(t)} < \frac{\cT_{\nu(t)+1 }}{\nu(t)+1}\cdot \frac{ \nu(t)+1}{\nu(t)}  
\end{equation*} 
we get $ \lim _{t \to \infty} \nu(t)/t = 1/ \bbE (\t_i)$ a.s.  Since $Z_t/t= [W_{\nu(t) }/\nu(t)]\cdot[\nu(t)/t]$ we get the thesis.

\subsection{Proof of the Invariance Principle in Theorem \ref{teo1}}. For each $n\geq 1$ and $t\in \bbR_+$  let 
\[
 A^{(n)}_t := \frac{1}{\sqrt{n} } \left\{ W_{\lfloor nt \rfloor}- \bbE( w_i) nt \right\}\,,\qquad  D^{(n)}_t:= \frac{1}{\sqrt{n} }  \left\{ \cT_{\lfloor nt \rfloor}- \bbE( \t_i) nt \right\}\, .
\]
Then, since $v= \bbE(w_i)/ \bbE(\t_i)$,  the following identity holds:
  \begin{equation} \label{siriano2}
 \begin{split}  B^{(n)}_{t}& =\frac{1}{\sqrt{n}} \left\{
 W_{\nu(nt)}- \bbE(w_1) \nu(nt) 
 \right\}+\frac{1}{\sqrt{n}} \left\{
 \bbE(w_1) \nu(nt) - v \cT_{\nu(nt)}
 \right\}
 + \frac{v}{\sqrt{n}} \left\{
  \cT_{\nu(nt)}- nt 
  \right\} \\
  &
 =
 A^{(n)}_{\nu(nt)/n }- v D^{(n)} _{\nu(nt)/n}
 + \frac{v}{\sqrt{n}} \left\{
  \cT_{\nu(nt)}- nt 
  \right\} \, . 
  \end{split}
 \end{equation}
 \begin{Lemma}\label{pecorino}
 Given $\e>0$ and $s>0$ we have
 \begin{equation}\label{stellina}
\lim_{n\to \infty} \bbP \left( \sup_{0\leq t \leq s}  \left| \frac{1}{\sqrt{n}} 
 \left\{\cT_{ \nu( n t ) } -nt 
   \right\}
\right|>\e \right)=0\,.
 \end{equation}
 \end{Lemma}
 \begin{proof} 
 As proven in the previous subsection $\lim _{t \to \infty} \nu(t)/t= \theta:= 1/ \bbE(\t_i) $ a.s. Hence, fixed $\delta >0$, it holds $\bbP( E_n^c) \leq \d$ for $n$ large enough as we assume, where $E_n$ is the event $\{\nu( n s ) \leq 2\theta n s \}$. Trivially the event $E_n$ implies that 
$ \nu(nt)\leq  2\theta n s $ for each $t \in [0,s]$. Now we observe that, due to \eqref{warwick},   $\cT_{\nu(nt)}\leq nt < \cT_{\nu(nt)+1}$, hence $0\leq nt -\cT_{\nu(nt)} \leq \t_{\nu(nt)+1}$. 
Due to the above considerations,
calling $F_n$ the event in 
\eqref{stellina} we conclude that
\begin{equation}\label{luna}
\bbP(F_n) 
 \leq 
 \bbP(E_n^c)+ \bbP \left( \max _{ 1 \leq i \leq 2\theta n s+1 } \tau_i  > \e \sqrt{n}\right)
\end{equation}
Since $\bbP(E_n^c) \leq \d$ eventually,  by the arbitrariness of $\delta$ we only need to show that the last probability in \eqref{luna} goes to zero as $n \to \infty$. This is a general  fact. Let $(X_i)_{i \geq 1}$ be i.i.d. positive random variables with $ \bbE(X_i^2) < \infty$ (in our case $X_i= \t_i$). Then, given $a >0$, 
\begin{multline*}
\bbP(\max _{1\leq i \leq N} X_i \leq a \sqrt{N})=
 \bbP( X_1\leq    a \sqrt{N})^N =[ 1- \bbP( X_1>   a \sqrt{N})]^N \\
 =  e^{N \ln[ 1- \bbP( X_1>   a \sqrt{N})] } \sim e^{- N \bbP( X_1>   a\sqrt{ N})}\,. $$
 \end{multline*}
 Note that 
 the last equivalence holds since $\lim_{N \to \infty}\bbP( X_1>   a\sqrt{ N})=0$. At this point, in order to  prove that  $\max _{1\leq i \leq N} X_i / \sqrt{N}$ weakly converges to zero we only need to show   that $\lim _{N \to \infty} N \bbP( X_1>   a\sqrt{ N}) =0$, or equivalently that
 $ \lim _{t \to \infty} t^2 P(X_1>t)=0$. This follows form the fact that $\bbE(X_i^2) < \infty$ (see 
 Exercise 3.5, page 15 of \cite{Du}).
 \end{proof}
Due to Lemma  \ref{pecorino} we can disregard  the last addendum in \eqref{siriano2}   in order to prove the invariance principle for $ B^{(n)}$.  
Let us now consider the random path $\Gamma ^{(n)}$  in $D(\bbR_+; \bbR\times \bbR \times \bbR_+)$ defined as
$$  \Gamma^{(n)}: \bbR_+ \ni t \rightarrow 
\left( A^{(n)}_t ,D^{(n)}_t ,  \nu( n t )/n\right) \in  \bbR\times \bbR \times \bbR_+\,. $$
\begin{Lemma}
The random path $\G^{(n)}$ weakly converges to the random path
$$ \left(\,\left( B_1(t), B_2(t) , \theta t \right) \,\right)_{t \in \bbR_+}
 $$ 
 where $ \left( \,\left( B_1(t), B_2(t)\right ) \,\right)_{t \in \bbR_+}
$ is a zero mean bidimensional Brownian motion   such that  \begin{equation*}
 \begin{cases}
 {\rm Var}\bigl(B_1(1) \bigr) ={\rm Var} (w_1) \,,\\
 {\rm Var} \bigl( B_2(1) \bigr)={\rm Var}(\t_1)  \,,\\
  {\rm Cov}\bigl( B_1(1), B_2(1)  \bigr)={\rm Cov}(w_1,\t_1). \end{cases}
 \end{equation*}
\end{Lemma}
\begin{proof} 
We first  show that the random path  $ \bigl(\nu(n t )/n\bigr)_{t \in \bbR_+}$ weakly converges to the deterministic path $(\theta t )_{t \in \bbR_+}$. To this aim it is enough to show the convergence in probability w.r.t. the uniform distance on finite intervals (this implies the convergence in the Skohorod topology).
In particular, we claim that   for any $s , \delta >0$ it holds
\begin{equation}\label{colazione}
\lim _{n \to \infty} \bbP\left( \sup_{0 \leq t \leq s} \Big| \frac{  \nu(nt)}{n} - \theta t \Big| > \delta 
\right)=0\,.
\end{equation}
By monotonicity $\nu(n u)\leq \nu (n t) \leq \nu(n v)$ if $ u \leq t \leq v$, thus implying that 
\[ \Big| \frac{  \nu(nt)}{n} - \theta t \Big|\leq 
\max \left\{ \Big| \frac{  \nu(nu )}{n} - \theta u \Big|+\theta |u-t|\,, \, 
\Big| \frac{  \nu(nv )}{n} - \theta v\Big|+ \theta |v-t|\right\}\,.
\]
At this point the convergence  \eqref{colazione}  follows easily  from the 
 convergence in probability for fixed times, i.e. from the fact that $\nu(n r)/n \to \theta r$ a.s. for each fixed $r$.

Since by the invariance principle for sums of independent vectors it holds 
$$  \Big( A^{(n)}_t ,D^{(n)}_t  \Big ) _{t \in \bbR_+}\Longrightarrow \left( \, \bigl( B_1(t), B_2(t)\bigr ) \,\right)_{t \in \bbR_+}
$$
and since we have just proved that  $ \bigl(\nu(nt)/n\bigr)_{t \in \bbR_+} \Longrightarrow\bigl( \theta t \bigr)_{t \in \bbR_+}$ (where the r.h.s. is a deterministic path), the thesis follows from Theorem 3.9 of \cite{B}.

\end{proof}
We can now conclude the proof of the invariance principle. 
Since the map $ \bbR_+ \ni t \to A^{(n)}_{ \nu(nt)/n } \in \bbR$ if simply the composition $f \circ g (t)$ where $f(u)= A^{(n)}_{u } $ and $g(t)= \nu(nt)/n$, and similarly for  $ \bbR_+ \ni t \to D^{(n)}_{ \nu(nt)/n } \in \bbR$, we can simply combine the above lemma with   the Lemma in \cite{B}[page 151]  to derive the weak convergence
\begin{equation}
\Big( \,A^{(n)}_{ \nu(nt)/n } - v D^{(n)}_{ \nu(nt)/n } \,\Big) _{t \in \bbR_+} \; \Longrightarrow 
\; \Big( \,B_1(\theta t) -v B_2(\theta t) \, \Big)_{t \in \bbR_+}\,.
\end{equation}
Since the last process is a Brownian motion with diffusion constant \eqref{diffondo}, combining the above convergence with \eqref{siriano2} and Lemma \ref{pecorino} we get the invariance principle for $ B^{(n)}$.


\medskip

\medskip

\noindent {\bf Acknowledgements.}   V. Silvestri thanks the Department of Mathematics in University ``La Sapienza''
for the hospitality and  acknowledges the support of the UK Engineering and Physical Sciences Research Council (EPSRC) grant EP/H023348/1 for the University of Cambridge Centre for Doctoral Training, the Cambridge Centre for Analysis.

\end{document}